\documentclass[1996/10/24]{amsart}

\usepackage{amssymb,stmaryrd}
\usepackage{lipsum}
\usepackage{amsfonts}
\usepackage{graphicx}
\usepackage{epstopdf}
\usepackage{algorithmic}
\usepackage{enumerate}
\usepackage{bm}
\usepackage{multirow}
\usepackage{color}

\allowdisplaybreaks[4]

\ifpdf
  \DeclareGraphicsExtensions{.eps,.pdf,.png,.jpg}
\else
  \DeclareGraphicsExtensions{.eps}
\fi

\newtheorem{remark}{Remark}
\newtheorem{theorem}{Theorem}[section]
\newtheorem{lemma}{Lemma}[section]
\newtheorem{assum}{Assumption}
\newtheorem{define}{Definition}

\def\bmalpha{\boldsymbol{\alpha}}
\def\bmbeta{\boldsymbol{\beta}}
\def\bmr{\boldsymbol{r}}
\def\bmz{\boldsymbol{z}}
\def\bmPhi{\boldsymbol{\Phi}}

\def\calS{\mathcal{S}}

\def\red{\textcolor{red}}
\def\widep{\widetilde{p}}
\def\widemu{\widetilde{\mu}}

\usepackage[a4paper, left=2.3cm, right=2.3cm, top=3cm, bottom=3cm]{geometry}

\usepackage{fancyhdr}

\begin{document}

\title[Analysis of two-phase flow in porous media]
{The Existence, uniqueness, and regularity of weak solutions for a thermodynamically consistent two-phase flow model in porous media} 
\thanks{Huangxin Chen was supported by National Key Research and Development Project of China (Grant No. 2023YFA1011702) and National Natural Science Foundation of China (Grant No. 12471345, 12122115). Haitao Leng was supported by Basic and Applied Basic Research Foundation of Guangdong Province (Grant No. 2024A1515011163, 2025A1515010428)}


\author{Huangxin Chen}
\address{School of Mathematical Sciences and Fujian Provincial Key Laboratory on Mathematical Modeling and High Performance Scientific Computing, Xiamen University, Fujian, 361005, China}
\email{chx@xmu.edu.cn}

\author{Jisheng Kou}
\address{School of Civil Engineering, Shaoxing University, Shaoxing 312000, Zhejiang, China; School of Mathematics and Statistics, Hubei Engineering  University, Xiaogan 432000, Hubei, China}
\email{jishengkou@163.com}

\author{Haitao Leng}
\address{Corresponding authour. School of Mathematic and Information Sciences, Guangzhou University, Guangzhou, 510006, Guangdong, China} \email{htleng@m.scnu.edu.cn}

\author{Shuyu Sun}
\address{School of Mathematical Sciences, Tongji University, Shanghai 200092, China}
 \email{suns@tongji.edu.cn}

\author{Hai Zhao}
\address{School of Mathematical Sciences and Fujian Provincial Key Laboratory on Mathematical Modeling and High Performance Scientific Computing, Xiamen University, Fujian, 361005, China}
\email{zhaohai@stu.xmu.edu.cn}

\subjclass[2010]{}

\keywords{existence and uniqueness; regularity;
thermodynamically consistent model; two-phase flow; porous media}

\date{}

\begin{abstract}
Thermodynamically consistent models for two-phase flow in porous media have attracted significant attention in recent years.
In this paper, we prove the existence, uniqueness and regularity of the weak solution to such a recent model proposed in \cite{Gao2020,Kou2022}. To this end, firstly, we introduce a fully implicit time semi-discrete approximation and a fully discrete approximation for an appropriate weak formulation of the thermodynamically consistent model.
Next, by using the zeros of a vector field theorem, we prove the existence of the weak solution for the fully discrete approximation.
Then the existence of weak solutions for the fully implicit time semi-discrete approximation and the weak formulation of the model are derived by the weak convergence technique and the energy stability estimate.
Subsequently, by the Gr{\" o}nwall inequality,  we prove the uniqueness result under the smoothness assumption on the chemical potential.
Finally, combined with the regularity theory of elliptic partial differential equations (PDE), the regularity of the weak solution for the model with complete Neumann boundary conditions is established. 
\end{abstract}




%
%
%

%

\maketitle

\section{Introduction}

The model of two-phase flow in porous media is of great importance in geo-energy recovery and groundwater  management \cite{Chavent1986,Chen2006}.
 For instance, during the secondary oil recovery for petroleum extraction, water is injected to displace  oil in the reservoir.
Due to its importance,   modeling and numerical simulation of two-phase flow in porous media have been extensively studied \cite{Arbogast1996,Chen2006,Kou2014,Chen2019,Hoteit2008,HoteitH2008}.

The second law of thermodynamics is acknowledged as a general and fundamental  principle and reliable  mathematical models should obey this law.  Such models can be called thermodynamically consistent mathematical models.  For a specific isothermal fluid flow problem, the second law of thermodynamics yields a certain energy dissipation law, which actually provides a powerful mathematical tool to analyze the model as depicted in this paper (see the proof of Lemma \ref{lemma 3.4.1}).
The thermodynamically consistent models have been extensively studied in the field of two-phase diffuse interface flows \cite{Abels2012,Guo2015,Zhu2019,Feng2012,Shen2014,Cueto2014,Chen2015,Zhu2020}. Thermodynamically consistent approaches have also been explored to interpret the classical two-phase flow models in porous media in the literature, for instance \cite{Cances2018,Hassanizadeh2011} and the references therein.
In recent years, Kou et al. \cite{Gao2020,Kou2022} proposed a novel thermodynamically consistent model for incompressible and immiscible two-phase flow in porous media by introducing logarithmic free energy to characterize the capillarity effect.
As a result, the second law of thermodynamics is naturally obeyed by the new model. Moreover, the new model has elegant symmetry in its mathematical structure.
Meanwhile, relative permeability and capillarity effects arising from the classical models of two-phase flow in porous media are also inherent in the new model.
Energy stable numerical methods for the new model have been developed in \cite{Kou2022,Kou2023}, and its extensions to the two-phase flow coupled with rock compressibility, the two-phase flow in poro-viscoelastic media, and unsaturated flows in porous media have been explored in \cite{KouJ2023,KouJ2024_0,KouJ2024} as well.
Nevertheless, there is a lack of knowledge about the existence, uniqueness and regularity of the weak solution for the new model. Hence, we devote to this issue in this paper.

Due to the highly nonlinear nature of the two-phase flow, there is a significant challenge to the mathematical analysis of the weak solution.
The classical model of two-phase flow in porous media has been well studied and the existence of the weak solution under assumptions on physical data has been proved.
In the one-dimensional case, there are some mathematical analysis \cite{Bertsch2003,Buzzi2009,Cances2009,CancesC2009,Steinle2022} that ensure the well-posedness of the problem.
In the higher dimensional spaces, the existence of the weak solution for incompressible and immiscible two-phase flow in an unfractured, single-porosity and dual-porosity porous media has been shown in  \cite{Alt1985,Amadori2015,Arbogast1992,Kroener1984}.
Chen \cite{Chen2001,Chen2002} analyzed the existence and regularity of the weak solution for the degenerate two-phase incompressible flow model in porous media and established the uniqueness of the weak solution under the assumption that the artificial global pressure is Lipschitz continuous in space.
The validity of this assumption, however, requires that the saturation is H\"{o}lder continuous in space.
For heterogeneous porous media comprising two distinct rock types, the existence result for the weak solution of the two-phase flow model with discontinuous capillary pressure field has been proved in \cite{Cances2012}.
More works on the existence of weak solutions for the classical models of two-phase flow in porous media can be found in \cite{Bourgeat1995,Amaziane1996,Daim2007,Mikelic2010}.

Most of the works established the existence of the weak solution for the classical models of two-phase flow in porous media by introducing the artificial and complementary pressures and assuming that the boundedness of higher-order derivative terms holds.
However, the complementary pressure has no physical meaning and is difficult to inversely solve the saturation by it. Therefore,
in this paper, we prove the existence of a weak solution to the new model by taking the chemical potential and the real phase pressure as the primary variables.
Compared with \cite{Arbogast1992,Chen2001}, there is no need for any assumption on the boundedness of higher-order derivative terms.

Based on the Galerkin approximation method \cite{Evans2022}, we establish for the first time the existence,  uniqueness and regularity of the weak solution to the new model. 
Instead of using the regularized problem, we primarily employ the zeros of a vector field theorem \cite{Evans2022}, which was utilized in \cite{Arbogast1992} for addressing nonlinear differential equations, to obtain the existence of the weak solution for the fully discrete scheme in this paper.
Then we obtain the existence of the weak solution for the fully implicit time semi-discrete scheme and demonstrate the boundedness of the saturation and the energy stability estimate.
Subsequently, we prove the existence of the weak solution for the new model.
In order to prove the uniqueness result, it suffices to assume that the chemical potential is Lipschitz continuous in space and the associated Lipschitz constant belongs to $L^2$ in the time.
Based on Theorem 2.5 in \cite{Chen2002}, we introduce the artificial and complementary pressures to obtain an elliptic equation for the artificial pressure and a parabolic equation for the complementary pressure.
Then we can deduce the regularity of the weak solution for the model with complete Neumann boundary conditions by using the regularity theory of elliptic PDEs.

The paper is organized as follows.
In Section 2, we describe the thermodynamically consistent model for incompressible and immiscible two-phase flow in porous media and give a weak formulation of the new model. Then we present the existence result of a weak solution to the new model.
In Section 3, we introduce a fully implicit time semi-discrete approximation of the new model and then establish a fully discrete approximation.
In Section 4, we first establish the existence of the weak solution for the fully discrete scheme, then prove the existence of the weak solution for the fully implicit time semi-discrete scheme, and finally demonstrate the existence of the weak solution to the new model.
The uniqueness and regularity of the weak solution are analyzed in Sections 5 and 6, respectively. 
Some concluding remarks are given in the last section.

\section{Mathematical model and preliminary}
Let $\Omega$ be an open and bounded domain in $\mathbb{R}^d~(d=2,3)$ with the Lipschitz boundary $\partial \Omega$ and $T$ a positive constant denoting the final time. We consider the following model \cite{Kou2022} of two-phase flow in porous media $\Omega$:
 \begin{align}
 \phi\frac{\partial {S_{\alpha}}}{\partial t} +\nabla \cdot \mathbf{u}_{\alpha}
  = q_{\alpha},&  \ \ {\alpha}=w,\ n,
   \label{F1.1} \\
\mathbf{u}_{\alpha} = -\lambda_{\alpha}\bm{\mathcal{K}}\nabla(\widep+\widemu_{\alpha}),&  \ \ {\alpha}=w,\ n,
 \label{F1.2} \\
 \widemu_{\alpha}=\frac{\partial F(S_w,S_n)}{\partial S_{\alpha}},&  \ \  {\alpha}=w,\ n,  \label{F3} \\
        S_w+S_n=1,&     \label{F1.4}	
\end{align}
where $\alpha=w, n$ denote the wetting and non-wetting phases respectively, $\phi=\phi(x)$ is the porosity of the porous medium. $S_{\alpha}$, $\mathbf{u}_{\alpha}$, $q_{\alpha}$, and $\widemu_{\alpha}$ represent the saturation, the Darcy's velocity, the external volumetric flow rate, and the chemical potential of the $\alpha$-phase, $\widep$ and $F$ represent the pressure and the free energy function.
The tensor $\bm{\mathcal{K}}$, which is bounded, symmetric, and uniformly positive definite, represents the absolute permeability, i.e. $$
K_{\min}\bm{x}\cdot\bm{x}\leq \bm{\mathcal{K}}\bm{x}\cdot\bm{x}\leq K_{\max}\bm{x}\cdot\bm{x},\quad  \forall \bm{x}\in \mathbb{R}^d,
$$
where $K_{\min}$ and $\ K_{\max}$ are two positive constants.
Moreover, $\lambda_{\alpha}=\frac{k_{r\alpha}}{\eta_{\alpha}}$ denotes the phase mobility, $k_{r\alpha}$ and $\eta_{\alpha}$ are the relative permeability and viscosity of the $\alpha$-phase, respectively.
In practice, there usually exists a certain amount of residual fluids, so we assume in this paper that $S_{\epsilon}\leq S_w\leq 1-S_{\epsilon}$, where $S_\epsilon$ (a sufficiently small positive constant) denotes the residual saturation level.

In the paper, for any $1\leq p\leq \infty$, nonnegative integer $s$ and an open subset $\mathcal{D}\subset \Omega$, we let $W^{s,p}(\mathcal{D})$ be the standard Sobolev spaces (see \cite{Adams2003}) with the norms $\Vert\cdot\Vert_{W^{s,p}(\mathcal{D})}$. In particular, if $p=2$, $W^{s,2}(\mathcal{D})$ is denoted by $H^s(\mathcal{D})$. If $s=0$, $W^{0,p}(\mathcal{D})$ is denoted by $L^p(\mathcal{D})$. Moreover, we will also use the dual space $H^{-1}(\Omega)$ of $H^1(\Omega)$ with the norm $\Vert\cdot\Vert_{H^{-1}(\Omega)}$.
For any $\beta\in (0,1]$, let $C^{s,\beta}(\Omega)$ be the standard H{\" o}lder spaces (see \cite{Adams2003}) with the norms $ \Vert \cdot \Vert_{C^{s,\beta}(\Omega)}$.
For any Banach spaces $X$, $L^p(0,T;X)$ and $W^{1,p}(0,T;X)$ denote the Bochner spaces (see \cite{Hytonen2016}) with the norms
\begin{align*}
\Vert v\Vert_{L^p(0,T;X)}^p=\int_0^T \Vert v(t)\Vert_{X}^p\mathrm{d}t, \ 1\leq p<\infty,\quad \ \Vert v\Vert_{L^\infty(0,T;X)}=\max\limits_{0\leq t\leq T} \Vert v(t)\Vert_{X}, \\
\Vert v\Vert_{W^{1,p}(0,T;X)}^2=\Vert v\Vert_{L^p(0,T;X)}^2+\Vert\partial_t v\Vert_{L^p(0,T;X)}^2.
\end{align*}

\begin{assum}
\label{assum 1}
Let $\phi\in L^{\infty}(\Omega)$ satisfy $\phi(x)\geq \phi_m>0$, where $\phi_m$ is a positive constant. We assume that $q_{\alpha}, k_{r\alpha}\in L^\infty(0,T;L^\infty(\Omega))$ only depend on $S_w$ and are continuous with respect to $S_w$. Furthermore, we assume $k_{rw}(0)=0,\ k_{rw}(S_w)>0$ for $S_w>0$ and $k_{rn}(1)=0,\ k_{rn}(S_w)>0$ for $S_w<1$.
\end{assum}

From Assumption \ref{assum 1}, we can obtain
$$
\lambda_{\min}\leq \lambda_\alpha \leq \lambda_{\max}, \quad \alpha=w,\ n,
$$
for $S_{\epsilon}\leq S_w\leq 1-S_{\epsilon}$, where $\lambda_{\min}$ and $\lambda_{\max}$ are two positive constants.

Next, We define the general phase potential $p_{\alpha}$ and the capillary pressure $p_{c}$ as follows:
\begin{equation} \label{F7}
p_{\alpha}=\widep+\widemu_{\alpha}~~ \  {\alpha}=w,\ n, \quad\quad
p_{c}=p_{n}-p_{w}=\widemu_{n}-\widemu_{w}.
\end{equation}

The expression of the free energy function $F$ defined as that in \cite{Kou2022} is given by
\begin{equation} \label{F4}
F(S_w,S_n)= \sum_{{\alpha}=w, n} \gamma_{\alpha}S_ {\alpha}({\rm{ln}} \ (S_ {\alpha})-1)+ \gamma_{wn}S_w S_n,
\end{equation}
and its equivalent formulation based only on $S_w$ can be rewritten as:
\begin{equation} \label{F10}
F(S_w)=\gamma_w S_w ({\rm{ln}} \ (S_w)-1)+\gamma_n(1-S_w)({\rm{ln}}(1-S_w)-1)+\gamma_{wn}S_w(1-S_w),
\end{equation}
where $\gamma_{\alpha}>0$ for $\alpha=w, n, wn$ are the energy parameters. Then the chemical potential denoted by $\mu_w(S_w)$ and derived
from the derivative of $F(S_w)$ with respect to $S_w$ can be formulated as
\begin{equation} \label{F11}
\mu_w(S_w)=F'(S_w)=\gamma_w {\rm{ln}} \ (S_w)-\gamma_n{\rm{ln}} \ (1-S_w)+\gamma_{wn}(1-2S_w).
\end{equation}

\begin{assum}\label{assum 2}
 The energy parameters $\gamma_w$, $\gamma_n$, and $\gamma_{wn}$ belong to $L^\infty (\Omega)$ and satisfy that
\begin{equation}\label{eqassum 2}
\frac{\gamma_w}{x}+\frac{\gamma_n}{1-x}-2\gamma_{wn}\geq c_{\min}>0, \ \forall \ x\in(0,1),
\end{equation}
where $c_{\min}$ is a positive constant. By a simple calculation, we know that the minimum point of the left-hand side of \eqref{eqassum 2} is $x=\sqrt{\gamma_w}/(\sqrt{\gamma_w}+\sqrt{\gamma_n})$. Thus the condition \eqref{eqassum 2} is equivalent to
\begin{equation}\label{eqassum 2.1}
\left(\sqrt{\gamma_w}+\sqrt{\gamma_n}\right)^2-2\gamma_{wn} \geq c_{\min}>0.
\end{equation}
Moreover, we observe that almost all the data used in \textup{\cite{Kou2022,KouJ2023}} satisfy the condition \eqref{eqassum 2.1}.
\end{assum}

By calculating the derivative for $\mu_w$ with respect to $S_w$, it is easy to observe
by Assumption \ref{assum 2} that
$$\frac{d\mu_w(S_w) }{dS_w}=\frac{\gamma_w }{S_w}+\frac{\gamma_n }{1-S_w}-2\gamma_{wn}>0,$$
and this implies that $\mu_w$ is a strictly monotone increase function with respect to $S_w$. Thus we can uniquely solve a $S_w$ for a given $\mu_w$ by \eqref{F11}, which means that we can define a solution operator $\mathcal{S}$ such that $S_w=\mathcal{S}(\mu_w)$. Let $p=p_n$, then we can deduce from $S_n=1-S_w$ an equivalent formulation of the model in the following
 {\begin{equation} \label{F13}
 \left\{
 \begin{aligned}
 S_w=&\ \calS(\mu_w),\quad \text{in}~\Omega\times (0,T],\\
 -\nabla \cdot \left(\lambda_t \bm{\mathcal{K}}\nabla p\right)=&\ \nabla \cdot \left(\lambda_w  \bm{\mathcal{K}}\nabla \mu_w\right)
 +q_t,\quad \text{in}~\Omega\times (0,T],\\
\phi\frac{\partial {S_w}}{\partial t} -\nabla \cdot \left(\lambda_w\bm{\mathcal{K}}\nabla \mu_w\right)=&\ \nabla \cdot \left(\lambda_w \bm{\mathcal{K}}\nabla p\right)+q_w,\quad \text{in}~\Omega\times (0,T],
 \end{aligned}		
 \right.
\end{equation}}
where $q_t=q_w+q_n$ and $\lambda_t=\lambda_w+\lambda_n$ denotes the total mobility. Next, we close the model by providing the boundary and initial conditions:
 \begin{equation} \label{F13.1}
 \begin{cases}
 \mu_w = \varphi_1(\bm{x},t), &\text{on}~ \Gamma_1\times (0,T],\\
\lambda_n\bm{\mathcal{K}}\nabla p_n \cdot \mathbf{n} = \varphi_2, &\text{on}~ \Gamma_2\times (0,T],\\
 p=\varphi_3(\bm{x},t), &\text{on}~\Gamma_1\times (0,T],\\
\lambda_w\bm{\mathcal{K}}\nabla p_w\cdot \mathbf{n}=\varphi_4, &\text{on}~ \Gamma_2\times (0,T],
 \end{cases}		
\end{equation}
and
\begin{equation} \label{F13.2}
 \mu_w(x,0)=\mu_w^0(x), \quad \text{in}~ \Omega,
\end{equation}
where $\mathbf{n}$ is the outer unit normal vector to $\partial \Omega$,  $\Gamma_1\cup \Gamma_2=\partial \Omega$ and {$\Gamma_1\cap \Gamma_2=\emptyset$}, and $\mu_w^0\in L^2(\Omega)$ is a given function. In the following, we set $S_w^0=\mathcal{S}(\mu_w^0)$.

\begin{assum}
\label{assum 3}
 We assume that $\varphi_\alpha$ for $\alpha=1, 3$ can be extended to the whole domain $\Omega$  (see Definition 13.2 in {\cite{Tartar2007}}) such that
\begin{align*}
 \varphi_1 \in H^1(0,T;H^{-1}(\Omega))\cap L^2(0,T;H^{1}(\Omega)) \cap W^{1,1}(0,T;L^1(\Omega)), \\
   \varphi_3 \in L^2(0,T;H^{1}(\Omega)).
\end{align*}
On the other hand, we assume that $\varphi_\alpha\in L^\infty(0,T;L^\infty(\Gamma_2))$ for $\alpha=2, 4$ depend only on $S_w$ and are continuous with respect to $S_w$.
\end{assum}

Now, we define the spaces
 $$
 W=\{v\in H^1(\Omega):v|_{\Gamma_1}=0\},\quad \ V=\{v\in H^1(\Omega):v|_{\Gamma_1}=0; \ {\textup{if meas}\ (\Gamma_1)=0}\ \textup{then}\ \int_{\Omega} v \mathrm{d}x=0\}.
 $$
 \begin{remark}
     From a physical perspective, setting a Dirichlet boundary condition for the saturation $S_w$ at the fluid outlet is unreasonable. Therefore, through the constitutive relationship between the saturation $S_w$ and the chemical potential $\mu_w$, a Dirichlet boundary condition for the chemical potential $\mu_w$ should not be set at the fluid outlet. Hence, if $\textup{meas}\ (\Gamma_1)\neq 0$, the boundary $\Gamma_1$ should be designated as the fluid inlet boundary.
 \end{remark}
Then we give the following weak formulation for problems \eqref{F13}-\eqref{F13.2}.
 \begin{define}\label{define1}
 A triplet $(S_w,\mu_w,p)$ is a weak solution of the model \eqref{F13} with the boundary \eqref{F13.1} and initial \eqref{F13.2} if the saturation $S_w$ satisfies the conditions $S_\epsilon\leq S_w\leq 1-S_\epsilon$ and
 $$\mu_w \in \varphi_1(x,t)+L^2(0,T;W), \quad p\in \varphi_3(x,t)+L^2(0,T;V), \quad \partial_t S_w\in L^2(0,T;H^{-1}(\Omega)),$$
  such that
 \begin{equation} \label{equF20}
 \left\{
 \begin{aligned}
 &S_w=\calS(\mu_w)\ \text{a.e. in }\ [0,T]\times \Omega, \\
&\int_0^T (\lambda_t \bm{\mathcal{K}}\nabla p, \nabla v) \mathrm{d}t
 \\
= &\ -\int_0^T (\lambda_w \bm{\mathcal{K}}\nabla \mu_w,\nabla v)\mathrm{d}t+\int_0^T (q_t, v)\mathrm{d}t
+\int_0^T \langle\varphi_{2,4},v\rangle_{\Gamma_2}\mathrm{d}t,\ \forall v\in L^2(0,T;V), \\
& \int_0^T \left(\phi \partial_t S_w,\xi\right)\mathrm{d}t
+\int_0^T (\lambda_w\bm{\mathcal{K}}\nabla \mu_w,\nabla \xi) \mathrm{d}t
\\
=&\ -\int_0^T (\lambda_w\bm{\mathcal{K}}\nabla p,\nabla \xi)\mathrm{d}t+{\int_0^T} (q_w,\xi)\mathrm{d}t
+\int_0^T \langle\varphi_4,\xi\rangle_{\Gamma_2} \mathrm{d}t, \ \forall \xi\in L^2(0,T;W),\\
 &\mu_w(x,0)=\ \mu_w^0(x) \ \text{a.e. in}\ \Omega,
 \end{aligned}	
 \right.	
\end{equation}
where $\varphi_{2,4}=\varphi_4+\varphi_2$, $\partial_t S_w=\frac{\partial S_w}{\partial t}$, $(f,g)=\int_{\Omega} fg \mathrm{d}x$ and $\langle f,g\rangle_{\Gamma_2}=\int_{\Gamma_2} fg \mathrm{d}s$.
\end{define}

To obtain that the saturation $S_w$ satisfies $S_\epsilon\leq S_w\leq 1-S_\epsilon$, we assume that the following conditions hold.
\begin{assum}\label{assum 4}
We assume that $S_\epsilon\leq S_w^0\leq 1-S_\epsilon$ and
 \begin{align}
\lambda_n(S_\epsilon)q_w(S_\epsilon)-\lambda_w(S_\epsilon)q_n(S_\epsilon)\geq &\ 0,
\label{eqassum 4.1} \\
\lambda_n(S_\epsilon)\varphi_4(S_\epsilon)-\lambda_w(S_\epsilon)\varphi_2(S_\epsilon)\geq &\ 0,
\label{eqassum 4.2} \\
\lambda_n(1-S_\epsilon)q_w(1-S_\epsilon)-\lambda_w(1-S_\epsilon)q_n(1-S_\epsilon)\leq &\ 0,
\label{eqassum 4.3} \\
\lambda_n(1-S_\epsilon)\varphi_4(1-S_\epsilon)-\lambda_w(1-S_\epsilon)\varphi_2(1-S_\epsilon)\leq &\ 0.
\label{eqassum 4.4}
\end{align}
 Moreover, we assume $\mu_w(S_\epsilon)\leq \varphi_1\leq \mu_w(1-S_\epsilon)$. $f(S_w)=f(S_\epsilon)$ if $S_w\leq S_\epsilon$ and $f(S_w)=f(1-S_\epsilon)$ if $S_w\geq 1-S_\epsilon$, where $f=\lambda_w, \lambda_n, q_w, q_n, \varphi_2, \varphi_4$.
\end{assum}

\begin{remark}
    Assumptions \ref{assum 1}-\ref{assum 4} are consistent with (A3)-(A6) in the reference \textup{\cite{Chen2001}}.
\end{remark}

\begin{theorem}\label{mainresult}
     Under Assumption \ref{assum 4}, let a triplet $(S_w,\mu_w,p)$ be a weak solution of \eqref{equF20}, then the saturation $S_w$ satisfies $S_\epsilon\leq S_w\leq 1-S_\epsilon$.
\end{theorem}
\begin{proof}
     Obviously, we only need to prove $M_{\min}=\mu_w(S_\epsilon)\leq\mu_w\leq \mu_w(1-S_\epsilon)=M_{\max}$ in \eqref{equF20}. Note $\lambda_\alpha(S_w)=\lambda_\alpha(S_\epsilon),\ q_\alpha(S_w)=q_\alpha(S_\epsilon),\ \alpha=w,n$, and $\varphi_\alpha(S_w)=\varphi_\alpha(S_\epsilon),\ \alpha=2,4$,
     when $\mu_w\leq M_{\min}$.
     We set $\xi=v=(\mu_w-M_{\min})^-=\min\{\mu_w-M_{\min},0\}\leq 0$ in \eqref{equF20} to find that
     \begin{align*}
 & \lambda_t(S_\epsilon)\int_0^T (\bm{\mathcal{K}}\nabla p, \nabla (\mu_w-M_{\min})^-) \mathrm{d}t
  +\lambda_w(S_\epsilon)\int_0^T (\bm{\mathcal{K}}\nabla \mu_w,\nabla (\mu_w-M_{\min})^-)\mathrm{d}t \\
=&\ \int_0^T (q_t(S_\epsilon), (\mu_w-M_{\min})^-)\mathrm{d}t
+\int_0^T \langle\varphi_{2,4}(S_\epsilon),(\mu_w-M_{\min})^-\rangle_{\Gamma_2}\mathrm{d}t, \\
& \int_0^T \left(\phi \partial_t S_w,(\mu_w-M_{\min})^-\right)\mathrm{d}t
+\lambda_w(S_\epsilon)\int_0^T (\bm{\mathcal{K}}\nabla \mu_w,\nabla (\mu_w-M_{\min})^-) \mathrm{d}t \\
=&\ -\lambda_w(S_\epsilon)\int_0^T (\bm{\mathcal{K}}\nabla p,\nabla (\mu_w-M_{\min})^-)\mathrm{d}t+{\int_0^T} (q_w(S_\epsilon),(\mu_w-M_{\min})^-)\mathrm{d}t \\
&\ +\int_0^T \langle\varphi_4(S_\epsilon),(\mu_w-M_{\min})^-\rangle_{\Gamma_2} \mathrm{d}t.
 \end{align*}
 Then we compare the above two equations to conclude that
\begin{align*}
 &\int_0^T \left(\phi \partial_t \calS(\mu_w), (\mu_w-M_{\min})^-\right) \mathrm{d} t
 +\frac{\lambda_n(S_\epsilon)}{\lambda_t(S_\epsilon)}
 \int_0^T \left(\lambda_w(S_\epsilon)\bm{\mathcal{K}}\nabla{\mu}_w,\nabla (\mu_w-M_{\min})^-\right) \mathrm{d} t  \\
 =&\
 \int_0^T \left(q_w(S_\epsilon)-\frac{\lambda_w(S_\epsilon)}{\lambda_t(S_\epsilon)}q_t(S_\epsilon),
 (\mu_w-M_{\min})^-\right)\mathrm{d} t
 +\int_0^T \left\langle\varphi_4(S_\epsilon)-\frac{\lambda_w(S_\epsilon)}{\lambda_t(S_\epsilon)}\varphi_{2,4}(S_\epsilon)
 ,(\mu_w-M_{\min})^-\right\rangle_{\Gamma_2} \mathrm{d} t.
 \end{align*}
Combining conditions \eqref{eqassum 4.1} and \eqref{eqassum 4.2} in Assumption \ref{assum 4}, we can find that the right-hand side of the above equation is non-positive. Moreover, for the left-hand side, we have (for the detailed derivation of the inequality \eqref{appendixA}, see Appendix A)
\begin{equation}\label{appendixA}
 \int_0^T \left(\phi \partial_t \calS(\mu_w), (\mu_w-M_{\min})^-\right) \mathrm{d} t \geq 0,
\end{equation}
and
\begin{align*}
\int_0^T \left(\lambda_w(S_\epsilon)\bm{\mathcal{K}}\nabla \mu_w,\nabla (\mu_w-M_{\min})^-\right) \mathrm{d} t
= &\ \int_0^T \left(\lambda_w(S_\epsilon)\bm{\mathcal{K}}\nabla (\mu_w-M_{\min})^-,\nabla (\mu_w-M_{\min})^-\right) \mathrm{d} t
\geq  0.
\end{align*}
Hence $\mu_w\geq M_{\min}$. Following analogous reasoning, $\mu_w\leq M_{\max}$ can also be demonstrated. Then we conclude the proof.
\end{proof}

Now we present the main result as follows.

\begin{theorem}\label{thm 2.1}
Under Assumptions \ref{assum 1}-\ref{assum 4}, there exists a weak solution $(S_w,\mu_w,p)$ that satisfies the weak formulation defined in definition \ref{define1}. Moreover, $S_w, \mu_w \in C(0,T;L^2(\Omega)) \cap L^\infty(0,T;L^\infty(\Omega))$.
\end{theorem}

The existence of weak solutions in Theorem \ref{thm 2.1} will be proved in Section 4.
To the end, we use the Galerkin approximation method (see Chapter 7 in \cite{Evans2022}) and prove the existence of weak solutions for the fully discrete and the fully implicit time semi-discrete schemes presented in the next section.

\section{Approximation of the continuous model}

Throughout this paper, $c$ and $C$ will denote positive constants, which are not necessarily the same at different occurrences.

\subsection{Fully implicit time semi-discrete approximation}
 Firstly, we divide the time domain $(0,T)$ into $N$ equal subintervals with time step size $\tau=T/N$. Let $I_k=(t_k,t_{k+1}]~(k=0,\cdots,N-1)$,
where $t_k=k\tau$ for $k=0,\cdots,N-1$.
Since $\mu_w$ is a strictly monotone increase function with respect to $S_w$, there exist two positive constants $L_{\min}$ and $L_{\max}$ such that for any $S,\ \widetilde{S}\in \left[S_{\epsilon},1-S_{\epsilon}\right]$,
\begin{equation}\label{FderL}
\begin{aligned}
    L_{\min}(S-\widetilde{S})\leq \mu_{w}(S)-\mu_{w}(\widetilde{S})\leq L_{\max}(S-\widetilde{S}).
\end{aligned}
\end{equation}

\begin{define}\label{define2}
 For each integer $0\leq k\leq N-1$, we let $S_{\epsilon}\leq S_w^k\leq 1-S_{\epsilon}$. Then a triplet $(S_w^{k+1}, \mu_w^{k+1}, p^{k+1})$ is called a weak solution of the fully implicit time semi-discrete scheme if
$$\mu_w^{k+1}\in \varphi_1^{k+1}+W, \quad p^{k+1}\in \varphi_3^{k+1}+V $$
satisfy
 \begin{equation} \label{F21}
 \left\{
 \begin{aligned}
 &\left(\lambda_{t}^{k+1} \bm{\mathcal{K}}\nabla p^{k+1}, \nabla v\right) \\
 =&\ -\left(\lambda_w^{k+1} \bm{\mathcal{K}}\nabla \mu_w^{k+1}, \nabla v\right)+\left(q_t^{k+1},v\right)
 +\left\langle\varphi_{2,4}^{k+1},v\right\rangle_{\Gamma_2},\ \forall v\in V,\\
&\left(\phi\frac{S_w^{k+1}-S_w^{k}}{{\tau}},\xi\right)+
\left(\lambda_w^{k+1}\bm{\mathcal{K}}\nabla \mu_w^{k+1}, \nabla \xi\right)
 \\
= &\ -\left(\lambda_w^{k+1} \bm{\mathcal{K}}\nabla p^{k+1}, \nabla \xi\right)+\left(q_w^{k+1},\xi\right)
+\left\langle\varphi_4^{k+1},\xi\right\rangle_{\Gamma_2},\ \forall \xi\in W,
 \end{aligned}		
 \right.
\end{equation}
where $\varphi_\alpha^{k+1}=\frac{1}{\tau}\int_{I_k} \varphi_\alpha\mathrm{d}t$ for $\alpha\in\{1,3\}$, $\varphi_{2,4}^{k+1}=\varphi_{2}^{k+1}+\varphi_{4}^{k+1}$,
$$f^{k+1}=f\left(S_w^{k+1}\right),$$
for $f\in\{\lambda_{\alpha},q_{\alpha}\}$ with $\alpha\in\{w,n\}$ and $f=\varphi_{\alpha}$ with $\alpha\in\{2,4\}$, and $S_w^{k+1}= {\calS}(\mu_w^{k+1})$.
\end{define}

\begin{theorem}\label{thm semi3.2}
For each integer $0\leq k\leq N-1$, under Assumptions \ref{assum 1}-\ref{assum 4}, there exists a weak solution $(S_w^{k+1},\mu_w^{k+1},p^{k+1})$
 with $S_{\epsilon}\leq S_w^{k+1} \leq 1-S_{\epsilon}$ satisfying Definition \ref{define2}.
\end{theorem}

The proof of Theorem \ref{thm semi3.2} will be provided in Section 4.

\subsection{Fully discrete approximation}

Due to the density of $L^{\infty}(\Omega)\cap H^1(\Omega)$ in $H^1(\Omega)$, we can construct the finite-dimensional spaces $W_m \subset W \cap L^{\infty}(\Omega)$ and $V_m \subset V \cap L^{\infty}(\Omega)$ with the basis functions $\{\xi_i\}_{i=1}^{m}$  and $\{v_i\}_{i=1}^{m}$, respectively. Then
we obtain by replacing $V$ and $W$ with $V_m$ and $W_m$ in \eqref{F21} the following fully discrete scheme.
\begin{define}\label{define3}
For each integer $0\leq k\leq N-1$, let $S_{\epsilon}\leq S_{w}^k\leq 1-S_{\epsilon}$. Then a triplet $(S_{w,m}^{k+1}$, $\mu_{w,m}^{k+1}$, $p_m^{k+1})$ is a weak solution of the fully discrete approximation if
$$
\mu_{w,m}^{k+1}\in \varphi_1^{k+1}+W_m, \quad p_{m}^{k+1}\in \varphi_3^{k+1}+V_m
$$
satisfy
\begin{equation} \label{F23}
 \left\{
 \begin{aligned}
&\left(\lambda_{t,m}^{k+1} \bm{\mathcal{K}}\nabla p_{m}^{k+1}, \nabla v \right)\\
 =&\ -\left(\lambda_{w,m}^{k+1} \bm{\mathcal{K}}\nabla{\mu_{w,m}^{k+1}},\nabla v\right)
 +(q_{t,m}^{k+1},v)+\left<\varphi_{2,4,m}^{k+1},v\right>_{\Gamma_2},\ \forall \ v\in V_m,\\
 &\left(\phi\frac{S_{w,m}^{k+1}-S_{w}^k}{{\tau}}, \xi\right)+\left(\lambda_{w,m}^{k+1} \bm{\mathcal{K}}\nabla{\mu_{w,m}^{k+1}},\nabla \xi\right)
 \\
 =&\ -\left(\lambda_{w,m}^{k+1} \bm{\mathcal{K}}\nabla p_{m}^{k+1}, \nabla  \xi\right)
 +(q_{w,m}^{k+1},\xi)+\left<\varphi_{4,m}^{k+1},\xi\right>_{\Gamma_2},\ \forall \ \xi\in W_m,
 \end{aligned}		
 \right.
\end{equation}
where $\varphi_{2,4,m}^{k+1}=\varphi_{2,m}^{k+1}+\varphi_{4,m}^{k+1}$,
$$f_{m}^{k+1}=f(S_{w,m}^{k+1}),$$
for $f\in\{\lambda_{\alpha},q_{\alpha}\}$ with $\alpha\in\{w,n\}$ and $f=\varphi_{\alpha}$ with $\alpha\in\{2,4\}$, $S_{w,m}^{k+1}={\calS}(\mu_{w,m}^{k+1})$, and
\begin{equation} \label{231}
\mu_{w,m}^{k+1}=\varphi_1^{k+1}+\sum_{i=1}^{m} \alpha_{i}^{k+1} \xi_i, \ p_{m}^{k+1}=\varphi_3^{k+1}+\sum_{i=1}^{m} \beta_{i}^{k+1} v_{i}.
\end{equation}
\end{define}
For simplicity, we denote $\bm{\alpha}^{k+1}=(\alpha_1^{k+1}, \cdots, \alpha_m^{k+1})$, $\bmbeta^{k+1}=( \beta_1^{k+1}, \cdots, \beta_m^{k+1} )$ in the following.

\begin{theorem}\label{thm 3.1}
For each integer $0\leq k\leq N-1$, under Assumptions \ref{assum 1}-\ref{assum 4}, there exists a weak solution $(S_{w,m}^{k+1},\mu_{w,m}^{k+1},p_{m}^{k+1})$ satisfying Definition \ref{define3}.
\end{theorem}

The proof of Theorem \ref{thm 3.1} will also be provided in Section 4.

\section{The existence of weak solutions}
In this section, we will prove the existence of weak solutions for the fully discrete scheme, fully implicit time semi-discrete scheme and weak formulation defined in Definition \ref{define3}, Definition \ref{define2} and Definition \ref{define1} respectively.
\subsection{The proof of Theorem \ref{thm 3.1}}
At the beginning of this subsection, we introduce the following lemma.
\begin{lemma}\label{lemma 3.2}
\textup{(Zeros of a vector field, Lemma in Section 1 of Chapter 9 in \cite{Evans2022})} Assume that the continuous function $\bmPhi : \mathbb{R}^m \rightarrow \mathbb{R}^m$ satisfies $\bmPhi(\bmz)\cdot \bmz\geq 0$, \red{$(\forall\ \Vert \bmz\Vert=R$ for some $R>0)$}.
Then there exists a point $\bmz_0\in \{\bm{y}:\Vert \bm{y}\Vert\leq R\}$ such that $\bmPhi(\bmz_0)=0$.
\end{lemma}

Then we define a mapping $\bmPhi_m^{k+1} : \mathbb R^{2m} \rightarrow \mathbb R^{2m} $ {satisfying
 $$ \bmPhi_m^{k+1} (\bm{\alpha}^{k+1},\bm{\beta}^{k+1} ) =(  \bmr_{\mu}^{k+1}, \bmr_{p}^{k+1} ) , \quad \bmr_{\mu}^{k+1}=(r_{\mu,1}^{k+1}, \cdots, r_{\mu,m}^{k+1}), \quad  \bmr_{p}^{k+1}=(r_{p,1}^{k+1}, \cdots, r_{p,m}^{k+1}) ,   $$}
where for $i=1, \cdots, m$,
\begin{equation} \label{F24}
\begin{aligned}
r_{\mu,i}^{k+1}=&\ \left(\phi \frac{\calS (\mu_{w,m}^{k+1})-S_{w}^k}{{\tau}}, \xi_i\right)
+(\lambda_{w,m}^{k+1}\bm{\mathcal{K}}\nabla{\mu_{w,m}^{k+1}},\nabla \xi_i) \\
&+(\lambda_{w,m}^{k+1}\bm{\mathcal{K}}\nabla p_{m}^{k+1}, \nabla \xi_i )
-\left\langle\varphi_{4,m}^{k+1},\xi_i\right\rangle_{\Gamma_2}-(q_{w,m}^{k+1},\xi_i), \\
r_{p,i}^{k+1}=&\ (\lambda_{t,m}^{k+1}\bm{\mathcal{K}}\nabla p_{m}^{k+1}, \nabla v_i ) \\
&+(\lambda_{w,m}^{k+1}\bm{\mathcal{K}}\nabla{\mu_{w,m}^{k+1}}, \nabla v_i)
-\left\langle\varphi_{2,4,m}^{k+1},v_i\right\rangle_{\Gamma_2}-(q_{t,m}^{k+1},v_i).
\end{aligned}
\end{equation}

\begin{lemma} \label{lemma 3.1}
Suppose that Assumptions \ref{assum 1}-\ref{assum 4} hold.
Then the mapping $\bmPhi_m^{k+1}$ is continuous.
\end{lemma}
\begin{proof}
This lemma is a direct result of the fact that ${\calS}$, $\lambda_{\alpha}~(\alpha=w,n,t)$, $q_{\alpha}\ (\alpha=w,t)$ and $\varphi_\alpha \ (\alpha=2,4)$ are continuous.
\end{proof}

\textit{Proof of Theorem \ref{thm 3.1}.}
Given $m$, we follow from Lemma \ref{lemma 3.1} and Lemma \ref{lemma 3.2} that we only need to show
$$\bmPhi_m^{k+1}(\bmalpha^{k+1},\bmbeta^{k+1})\cdot (\bmalpha^{k+1},\bmbeta^{k+1})\geq 0,\quad \forall\ \Vert (\bmalpha^{k+1},\bmbeta^{k+1})\Vert=R,$$
 for some $R>0$. Under Assumptions \ref{assum 1}-\ref{assum 4} and the fact $a^2+b^2\leq 3a^2+3(a+b)^2$,  \eqref{F24} implies that
\begin{align*}
&\bmPhi_m^{k+1}(\bmalpha^{k+1},\bmbeta^{k+1})\cdot (\bmalpha^{k+1},\bmbeta^{k+1})= (\bmr_{\mu}^{k+1},\bmr_p^{k+1})\cdot (\bmalpha^{k+1},\bmbeta^{k+1}) \\
=&\ \left(\phi \frac{\calS (\mu_{w,m}^{k+1})-S_{w}^k}{{\tau}}, \xi\right)
+\left(\lambda_{w,m}^{k+1}\bm{\mathcal{K}}\nabla{\mu_{w,m}^{k+1}},\nabla \xi\right)
+\left(\lambda_{w,m}^{k+1}\bm{\mathcal{K}}\nabla p_{m}^{k+1}, \nabla \xi\right)\\
&-\left\langle\varphi_{4,m}^{k+1},\xi\right\rangle_{\Gamma_2}-(q_{w,m}^{k+1},\xi)
+\left(\lambda_{t,m}^{k+1}\bm{\mathcal{K}}\nabla p_{m}^{k+1}, \nabla v \right) \\
&+\left(\lambda_{w,m}^{k+1}\bm{\mathcal{K}}\nabla{\mu_{w,m}^{k+1}}, \nabla v\right)
-\left\langle\varphi_{2,4,m}^{k+1},v\right\rangle_{\Gamma_2}-(q_{t,m}^{k+1},v) \\
\geq &\ \left(\phi \frac{\calS (\xi)}{{\tau}}, \xi\right)
-C\left(\frac{1}{\tau}\Vert\phi\Vert_{L^2(\Omega)}
+\lambda_{\max}\Vert \varphi_1^{k+1}\Vert_{H^1(\Omega)}\right)\Vert \xi\Vert_{H^1(\Omega)} \\
&-C\left(\lambda_{\max}\Vert \varphi_3^{k+1}\Vert_{H^1(\Omega)}
+\Vert \varphi_{4,m}^{k+1}\Vert_{L^\infty(\Gamma_2)}+\Vert q_{w,m}^{k+1} \Vert_{H^{-1}(\Omega)} \right)\Vert \xi\Vert_{H^1(\Omega)} \\
&-C\left(2\lambda_{\max}\Vert \varphi_3^{k+1}\Vert_{H^1(\Omega)}
+\lambda_{\max}\Vert \varphi_1^{k+1}\Vert_{H^1(\Omega)}
+\Vert \varphi_{2,4,m}^{k+1}\Vert_{L^\infty(\Gamma_2)}+\Vert q_{t,m}^{k+1} \Vert_{H^{-1}(\Omega)}
\right)\Vert v\Vert_{H^1(\Omega)} \\
&+\left(\lambda_{w,m}^{k+1}\bm{\mathcal{K}}\nabla(\xi+v),\nabla (\xi+v)\right)
+\left(\lambda_{n,m}^{k+1}\bm{\mathcal{K}}\nabla v, \nabla v \right) \\
\geq &\  \left(\phi \frac{\calS (\xi)}{{\tau}}, \xi\right)
+ \frac{\lambda_{\min} K_{\min}}{3} (\Vert\nabla \xi\Vert_{L^2(\Omega)}^2
+\Vert\nabla v\Vert_{L^2(\Omega)}^2)-\varepsilon(\Vert v\Vert_{H^1(\Omega)}^2+\Vert \xi\Vert_{H^1(\Omega)}^2)-C,
\end{align*}
where $\xi=\sum_{i=1}^{m} {\alpha}_{i} \xi_i$, $v=\sum_{i=1}^{m}  {\beta}_{i}v_{i}$, and $\varepsilon$ is a positive constant that can be arbitrarily small. 
On the other hand, by a simple calculation, we have
\begin{align*}
\left(\phi \frac{\calS(\xi)}{{\tau}}, \xi\right)
=&\left(\phi \frac{\calS(\xi)-\calS (0)}{{\tau}}, \xi\right)
+\left(\phi \frac{{\calS}(0)}{\tau},\xi\right)\\
\geq& \frac{\phi_m}{\tau L_{\max}}\Vert \xi\Vert_{L^2(\Omega)}^2-\frac{C}{\tau}\Vert\phi\Vert_{L^2(\Omega)}^2-\frac{\varepsilon}{\tau}\Vert \xi\Vert_{L^2(\Omega)}^2,
\end{align*}
where the inequality \eqref{FderL} is used.
Hence, it holds that
\begin{equation} \label{Eq 3.1}
\bmPhi_m^{k+1}(\bmalpha^{k+1},\bmbeta^{k+1})\cdot (\bmalpha^{k+1},\bmbeta^{k+1})\geq c(\Vert v\Vert_{H^1(\Omega)}^2+\Vert \xi\Vert_{H^1(\Omega)}^2)-C\geq 0,
\end{equation}
for a sufficiently large number $R$. This completes the proof.    $\hfill \Box$

\begin{remark}\label{remark 3.5}
We note that the phase mobility $\lambda_\alpha\geq \lambda_{\min}>0$ for $\alpha=w,n$ is used here. Suppose that the boundary conditions are homogeneous, i.e. $\varphi_1=\varphi_3\equiv0$ and $\varphi_2=\varphi_4\equiv 0$. Then, if $\lambda_w\geq 0$ and $\lambda_n \geq \lambda_{\min}>0$, we can obtain that
\begin{align*}
\bmPhi_m^{k+1}(\bmalpha^{k+1},\bmbeta^{k+1})\cdot (\bmalpha^{k+1},\bmbeta^{k+1})\geq c(\Vert v\Vert_{H^1(\Omega)}^2+\Vert \xi\Vert_{L^2(\Omega)}^2)-C.
\end{align*}
Meanwhile, if $\lambda_n\geq 0$ and $\lambda_w \geq \lambda_{\min}>0$, we also obtain that
\begin{align*}
\bmPhi_m^{k+1}(\bmalpha^{k+1},\bmbeta^{k+1})\cdot (\bmalpha^{k+1},\bmbeta^{k+1})\geq c(\Vert v+\xi\Vert_{H^1(\Omega)}^2+\Vert \xi\Vert_{L^2(\Omega)}^2)-C.
\end{align*}
Hence, there exists a weak solution of the fully discrete scheme \eqref{F23} for the above two cases for a sufficiently large number $R$ as in the proof of Theorem \ref{thm 3.1}.
\end{remark}

\subsection{The proof of Theorem \ref{thm semi3.2}}
In order to establish the existence of a weak solution for the fully implicit time semi-discrete scheme defined in Definition \ref{define2}, we present the following uniform boundedness result according to the proof of Theorem \ref{thm 3.1} provided in the above.

\begin{lemma}\label{lemma 3.3}
For each integer $0\leq k\leq N-1$, let $(S_{w,m}^{k+1},\mu_{w,m}^{k+1},p_{m}^{k+1})$ be the weak solution  of the fully discrete scheme defined in Definition \ref{define3}. Then under Assumptions \ref{assum 1}-\ref{assum 4}, we have
\begin{equation} \label{Eq 3.2}
\Vert \mu_{w,m}^{k+1}-\varphi_1^{k+1}\Vert_{H^1(\Omega)}^2+\Vert p_{m}^{k+1}-\varphi_3^{k+1}\Vert_{H^1(\Omega)}^2\leq C,
\end{equation}
where $C$ does not depend on $m$ but is related to $\frac{1}{\tau}$.
\end{lemma}

From Lemma \ref{lemma 3.3}, it is easy to deduce that there is a subsequence (still labeled by $m$) $(\mu_{w,m}^{k+1}-\varphi_1^{k+1},p_{m}^{k+1}-\varphi_3^{k+1})$ such that there exists
a $(\mu_{w,\ast}^{k+1}, p_{\ast}^{k+1})\in W\times V$ satisfying
\begin{align}
\mu_{w,m}^{k+1}\rightharpoonup \mu_{w,\ast}^{k+1}+\varphi_1^{k+1}=\mu_{w}^{k+1},\quad
p_m^{k+1}\rightharpoonup p_{\ast}^{k+1}+\varphi_3^{k+1}=p^{k+1},\quad \text{weakly in}~H^1(\Omega).\label{WC1}
\end{align}
Then according to the Sobolev embedding theorem (see Lemma 8.3 in \cite{Gilbarg1977}), we know that
\begin{align}
\mu_{w,m}^{k+1}\rightarrow \mu_w^{k+1},\quad p_m^{k+1}\rightarrow p^{k+1},\quad \text{strongly in}~L^p(\Omega)~\text{for}~1\leq p<6,\label{SC1}
\end{align}
which, together with the definition of $\calS$ and the boundedness of $S_w^k$, leads to
\begin{align}
S_{w,m}^{k+1}=\calS (\mu_{w,m}^{k+1})\rightarrow \calS (\mu^{k+1})=S_w^{k+1},\quad \text{strongly in}~L^p(\Omega)~\text{for}~1\leq p<6.\label{SC2}
\end{align}

Now we start to prove Theorem \ref{thm semi3.2}.

\textit{Proof of Theorem \ref{thm semi3.2}.}
Firstly, we will verify that the limit $(S_w^{k+1}, \mu_w^{k+1}, p^{k+1})$ of the sequence $(S_{w,m}^{k+1}$, $\mu_{w,m}^{k+1}$, $p_m^{k+1})$ is a weak solution of the fully implicit time semi-discrete scheme defined in Definition \ref{define2}.
By letting $m\rightarrow \infty$ in \eqref{F23}, we obtain from the continuity of $q_{\alpha}~(\alpha=w,t)$ and $\varphi_{\alpha}~(\alpha=2,4)$ and \eqref{SC2} that
\begin{align*}
\left(\phi \frac{\calS(\mu_{w,m}^{k+1})-S_w^k}{{\tau}}, \xi_i\right)
\to &\ \left(\phi \frac{\calS (\mu_w^{k+1})-S_w^k}{{\tau}}, \xi_i\right),
\end{align*}
and
\begin{align*}
(q_{w,m}^{k+1},\xi_i)\to (q_w^{k+1},\xi_i), \quad&\
\left<\varphi_{4,m}^{k+1},\xi_i\right>_{\Gamma_2}
\to \left<\varphi_4^{k+1},\xi_i\right>_{\Gamma_2}, \\
(q_{t,m}^{k+1},v_i)\to (q_t^{k+1},v_i), \quad&\
\left\langle\varphi_{2,4,m}^{k+1},v_i\right\rangle_{\Gamma_2} \to \left\langle\varphi_{2,4}^{k+1},v_i\right\rangle_{\Gamma_2}.
\end{align*}
On the other hand, by a simple calculation, we have
\begin{align}
&\left(\lambda_{w,m}^{k+1}\bm{\mathcal{K}}\nabla\mu_{w,m}^{k+1},\nabla\xi_i\right)-\left(\lambda_{w}^{k+1}\bm{\mathcal{K}}\nabla\mu_{w}^{k+1},\nabla\xi_i\right)\nonumber\\
=&\left(\left(\lambda_{w,m}^{k+1}-\lambda_w^{k+1}\right)\bm{\mathcal{K}}\nabla\mu_{w,m}^{k+1},\nabla\xi_i\right)\label{thm 3.2-proof:1}\\
&+\left(\lambda_{w}^{k+1}\bm{\mathcal{K}}\nabla(\mu_{w,m}^{k+1}-\mu_w^{k+1}),\nabla\xi_i\right)=R_1^{k+1}+R_2^{k+1}.\nonumber
\end{align}
By the continuity of $\lambda_w$, \eqref{SC2} and the Lebesgue-dominated convergence theorem (see DCT in \cite{Van1968}), we obtain that
\begin{align*}
R_1^{k+1}\rightarrow 0,\quad \text{as}~m\rightarrow \infty.
\end{align*}
As for $R_2^{k+1}$, according to \eqref{WC1}, we get
\begin{align*}
R_2^{k+1}\rightarrow 0,\quad \text{as}~m\rightarrow \infty.
\end{align*}
Thus
\begin{align*}
\left(\lambda_{w,m}^{k+1}\bm{\mathcal{K}}\nabla\mu_{w,m}^{k+1},\nabla\xi_i\right)\rightarrow \left(\lambda_{w}^{k+1}\bm{\mathcal{K}}\nabla\mu_{w}^{k+1},\nabla\xi_i\right),\quad \text{as}~m\rightarrow \infty.
\end{align*}
Similarly, we have
\begin{align*}
 \left(\lambda_{w,m}^{k+1}\bm{\mathcal{K}}\nabla p_{m}^{k+1},\nabla\xi_i\right)\rightarrow \left(\lambda_w^{k+1}\bm{\mathcal{K}}\nabla p^{k+1},\nabla\xi_i\right),\quad \text{as}~m\rightarrow \infty, \\
 \left(\lambda_{t,m}^{k+1}\bm{\mathcal{K}}\nabla p_{m}^{k+1},\nabla\xi_i\right)\rightarrow \left(\lambda_t^{k+1}\bm{\mathcal{K}}\nabla p^{k+1},\nabla\xi_i\right),\quad \text{as}~m\rightarrow \infty.
\end{align*}
With all the ingredients obtained in the above, it is easy to find that $(\mu_w^{k+1},p^{k+1})$ is a weak solution of \eqref{F21}. Hence, combining $S_w^{k+1}=\calS(\mu_w^{k+1})$, we obtain that $(S_w^{k+1}, \mu_w^{k+1}, p^{k+1})$ is a weak solution of the fully implicit time semi-discrete scheme defined in Definition \ref{define2}.

Next, we will show $S_{\epsilon}\leq S_w^{k+1} \leq 1-S_{\epsilon}$.
    Similarly, we only need to prove $M_{\min}=\mu_w\left(S_{\epsilon}\right)\leq\mu_w^{k+1}\leq \mu_w\left(1-S_{\epsilon}\right)= M_{\max}$ in \eqref{F21}.
    Set $\xi=v=(\mu_w^{k+1}-M_{\min})^-=\min\{\mu_w^{k+1}-M_{\min},0\}\leq 0$ in \eqref{F21} and compare the two equations in \eqref{F21} to conclude that
\begin{align*}
 &\left(\phi \frac{{\calS} (\mu_w^{k+1})-S_w^k}{\tau}, (\mu_w^{k+1}-M_{\min})^-\right)
 +\frac{\lambda_n(S_\epsilon)}{\lambda_t(S_\epsilon)}
\left(\lambda_w(S_\epsilon)\bm{\mathcal{K}}\nabla{\mu}_w^{k+1},\nabla (\mu_w^{k+1}-M_{\min})^-\right)   \\
 =&\
 \left(q_w(S_\epsilon)-\frac{\lambda_w(S_\epsilon)}{\lambda_t(S_\epsilon)}q_t(S_\epsilon),
 (\mu_w^{k+1}-M_{\min})^-\right)
 +\left\langle\varphi_4(S_\epsilon)-\frac{\lambda_w(S_\epsilon)}{\lambda_t(S_\epsilon)}\varphi_{2,4}(S_\epsilon)
 ,(\mu_w^{k+1}-M_{\min})^-\right\rangle_{\Gamma_2}.
 \end{align*}
Combining conditions \eqref{eqassum 4.1} and \eqref{eqassum 4.2} in Assumption \ref{assum 4}, we can find that the right-hand side of the above equation is non-positive. Moreover, for the left-hand side, it is easy to see that
\begin{align*}
 \left(\phi \frac{{\calS} (\mu_w^{k+1})-S_w^k}{\tau}, (\mu_w^{k+1}-M_{\min})^-\right)
=&\ \left(\phi \frac{S_w^{k+1}-S_w^k}{\tau}, (\mu_w^{k+1}-M_{\min})^-\right) \geq 0,
\end{align*}
and
\begin{align*}
\left(\lambda_w(S_\epsilon)\bm{\mathcal{K}}\nabla \mu_w^{k+1},\nabla (\mu_w^{k+1}-M_{\min})^-\right)
= &\ \left(\lambda_w(S_\epsilon)\bm{\mathcal{K}}\nabla (\mu_w^{k+1}-M_{\min})^-,\nabla (\mu_w^{k+1}-M_{\min})^-\right)
\geq  0.
\end{align*}
Hence $\mu_w^{k+1}\geq M_{\min}$. Following analogous reasoning, $\mu_w^{k+1}\leq M_{\max}$ can also be demonstrated. Then we conclude the proof.     $\hfill \Box$

\begin{remark}
  If the phase mobility satisfies $\lambda_w\geq 0$ and $\lambda_n\geq \lambda_{\min}>0$ (or $\lambda_n\geq 0$ and $\lambda_w\geq \lambda_{\min} >0$), we can only establish a weak solution $(S_w^{k+1},\mu_w^{k+1},p^{k+1})$ in a less smooth space $L^2(\Omega)$ of \eqref{F21} with homogeneous boundary conditions, using Remark \ref{remark 3.5}.
   This is practical because it is possible that the chemical potential may lack continuity in situations where $\lambda_w=0$ or $\lambda_n=0$.
\end{remark}

\subsection{The proof of the existence for Theorem \ref{thm 2.1}}
 In this subsection, we will prove the existence of a weak solution in Theorem \ref{thm 2.1}.
Using the fact $S_w^{k} \in \left[S_{\epsilon},1-S_{\epsilon}\right]$ for every $k=0,\cdots,N$, we have the following energy stability estimate.
This property ensures that the weak solution of the fully implicit time semi-discrete scheme is uniformly bounded, which is essential for the convergence analysis.

\begin{lemma}\label{lemma 3.4.1}
For each integer $0\leq k\leq N-1$, let $(S_{w}^{k+1},\mu_w^{k+1},p^{k+1})$ be the weak solution of the fully implicit time semi-discrete scheme defined in Definition \ref{define2}. Then under Assumptions \ref{assum 1}-\ref{assum 4}, we have the following energy stability estimate:
\begin{equation} \label{Eq 3.3.1.1}
\begin{aligned}
&\ \left(\phi,F(S_w^{N})-F(S_w^0)\right)
+\frac{\phi_m c_{\min}}{2}\sum_{k=0}^{N-1} \left\Vert S_w^{k+1}-S_w^k\right\Vert_{L^2(\Omega)}^2
+\sum_{k=0}^{N-1} \tau\left( \lambda_n^{k+1} \bm{\mathcal{K}}\nabla p^{k+1},\nabla p^{k+1} \right) \\
&+\sum_{k=0}^{N-1} \tau\left(\lambda_w^{k+1} \bm{\mathcal{K}}\nabla\left( p^{k+1}+\mu_w^{k+1}\right),\nabla \left( p^{k+1}+\mu_w^{k+1}\right) \right)\\
\leq &\ \sum_{k=0}^{N-1} \left(S_w^{k+1}-S_w^{k},\phi\varphi_1^{k+1}\right)
+\sum_{k=0}^{N-1} \tau\left(\lambda_w^{k+1}\bm{\mathcal{K}}\nabla \left( p^{k+1}+\mu_w^{k+1}\right), \nabla \varphi_{1,3}^{k+1}\right) \\
&\ +\sum_{k=0}^{N-1} \tau\left(\lambda_n^{k+1}\bm{\mathcal{K}}\nabla p^{k+1}, \nabla \varphi_3^{k+1}\right)+\sum_{k=0}^{N-1} \tau\Big(
\left(q_n^{k+1},p^{k+1}-\varphi_3^{k+1}\right)
  +\left(q_w^{k+1},\mu_w^{k+1}+p^{k+1}-\varphi_{1,3}^{k+1}\right)
  \Big) \\
  &\ +\sum_{k=0}^{N-1} \tau\left(\left\langle\varphi_2^{k+1},p^{k+1}-\varphi_3^{k+1}\right\rangle_{\Gamma_2}+\left\langle\varphi_4^{k+1},\mu_w^{k+1}+p^{k+1}-\varphi_{1,3}^{k+1}\right\rangle_{\Gamma_2}
  \right),
\end{aligned}
\end{equation}
where $\varphi_{1,3}^{k+1}=\varphi_1^{k+1}+\varphi_3^{k+1}$.
\end{lemma}
\begin{proof}
 We take $v=p^{k+1}-\varphi_3^{k+1}$ in  \eqref{F21} to obtain that
 \begin{align*}
 &\left(\lambda_{t}^{k+1}\bm{\mathcal{K}}\nabla p^{k+1}, \nabla p^{k+1} \right)+\left(\lambda_w^{k+1} \bm{\mathcal{K}}\nabla \mu_w^{k+1}, \nabla p^{k+1}\right) \\
 =&\ \left(\lambda_{t}^{k+1}\bm{\mathcal{K}}\nabla p^{k+1}+\lambda_w^{k+1} \bm{\mathcal{K}}\nabla \mu_w^{k+1}, \nabla \varphi_3^{k+1} \right) +\left(q_t^{k+1},p^{k+1}-\varphi_3^{k+1}\right)
 +\left\langle\varphi_{2,4}^{k+1},p^{k+1}-\varphi_3^{k+1}\right\rangle_{\Gamma_2}.
 \end{align*}
 Selecting $\xi=\mu_w^{k+1}-\varphi_1^{k+1}$ in  \eqref{F21}, we can get that
 \begin{align*}
&\left(\phi\frac{S_w^{k+1}-S_w^{k}}{{\tau}},\mu_w^{k+1}\right)+
\left(\lambda_w^{k+1}\bm{\mathcal{K}}\nabla \mu_w^{k+1}, \nabla \mu_w^{k+1}\right)
+\left(\lambda_w^{k+1} \bm{\mathcal{K}}\nabla p^{k+1}, \nabla \mu_w^{k+1}\right) \\
= &\ \left(\phi\frac{S_w^{k+1}-S_w^{k}}{{\tau}},\varphi_1^{k+1}\right)
+\left(\lambda_w^{k+1}\bm{\mathcal{K}}\nabla (p^{k+1}+\mu_w^{k+1}), \nabla \varphi_1^{k+1}\right)
\\
&+\left(q_w^{k+1},\mu_w^{k+1}-\varphi_1^{k+1}\right)
+\left\langle\varphi_4^{k+1},\mu_w^{k+1}-\varphi_1^{k+1}\right\rangle_{\Gamma_2}.
 \end{align*}
Hence, it is easy to see by the above two equalities that
\begin{equation} \label{Eq 3.4.1.1}
\begin{aligned}
& \left(\phi\frac{S_w^{k+1}-S_w^{k}}{{\tau}},\mu_w^{k+1}\right)+\left( \lambda_w^{k+1} \bm{\mathcal{K}}\nabla \left(\mu_w^{k+1}+p^{k+1}\right),\nabla \left(\mu_w^{k+1}+p^{k+1}\right)\right)+\left( \lambda_n^{k+1} \bm{\mathcal{K}}\nabla p^{k+1},\nabla p^{k+1}\right) \\
= &\ \left(\phi\frac{S_w^{k+1}-S_w^{k}}{{\tau}},\varphi_1^{k+1}\right)
+\left(\lambda_w^{k+1}\bm{\mathcal{K}}\nabla \left(\mu_w^{k+1}+p^{k+1}\right), \nabla \varphi_{1,3}^{k+1}\right)
+\left(\lambda_n^{k+1}\bm{\mathcal{K}}\nabla p^{k+1}, \nabla \varphi_3^{k+1}\right) \\
&\ +\left(q_n^{k+1},p^{k+1}-\varphi_3^{k+1}\right)
 +\left(q_w^{k+1},\mu_w^{k+1}+p^{k+1}-\varphi_{1,3}^{k+1}\right) \\
&+\left\langle\varphi_2^{k+1},p^{k+1}-\varphi_3^{k+1}\right\rangle_{\Gamma_2}+\left\langle\varphi_4^{k+1},\mu_w^{k+1}+p^{k+1}-\varphi_{1,3}^{k+1}\right\rangle_{\Gamma_2}.
\end{aligned}
\end{equation}
On the other hand, by Appendix B, we have
\begin{equation} \label{Eq 3.4.1.2}
F(S_w^{k+1})-F(S_w^k)+\frac{c_{\min}}{2}\left(S_w^{k+1}-S_w^k\right)^2 \leq \left(S_w^{k+1}-S_w^k\right)\mu_w^{k+1}.
\end{equation}
Thus we can obtain the desired result by combining \eqref{Eq 3.4.1.1}-\eqref{Eq 3.4.1.2} and summing over $k$.
\end{proof}
\begin{remark}
If the system of two-phase flow in porous media is closed, i.e. $q_w=q_n\equiv 0$, $\varphi_1 (x,t)=\varphi_3 (x,t)\equiv 0$ and $\varphi_2 =\varphi_4 \equiv 0$, then  \eqref{Eq 3.3.1.1} can be transformed as (see Theorem 1 in \textup{\cite{KouJ2023}}):
\begin{align*}
 \left(\phi,F(S_w^{N})-F(S_w^0)\right)+\sum_{k=0}^{N-1} \tau\left(\left( \lambda_n^{k+1} \bm{\mathcal{K}}\nabla p^{k+1},\nabla p^{k+1} \right)  +\left(\lambda_w^{k+1} \bm{\mathcal{K}}\nabla\left( p^{k+1}+\mu_w^{k+1}\right),\nabla \left( p^{k+1}+\mu_w^{k+1}\right) \right) \right)
\leq  0.
\end{align*}
\end{remark}

Based on the energy stability estimate \eqref{Eq 3.3.1.1}, we have the following uniform boundedness result.
\begin{lemma}\label{lemma 3.4}
For each integer $0\leq k\leq N-1$, let $(S_{w}^{k+1},\mu_w^{k+1},p^{k+1})$ be the weak solution of the fully implicit time semi-discrete scheme defined in Definition \ref{define2}. Then under Assumptions \ref{assum 1}-\ref{assum 4}, we have
\begin{equation} \label{Eq 3.3.1}
\sum_{k=0}^{N-1} \left\Vert S_w^{k+1}-S_w^k\right\Vert_{L^2(\Omega)}^2+\sum_{k=0}^{N-1} \tau\Vert \mu_w^{k+1}-\varphi_1^{k+1}\Vert_{H^1(\Omega)}^2+\sum_{k=0}^{N-1} \tau\Vert p^{k+1}-\varphi_3^{k+1}\Vert_{H^1(\Omega)}^2\leq C,
\end{equation}
where $C$ does not depend on the time step ${\tau}$. Furthermore, there also exists a constant $C$ which does not depend on the time step size ${\tau}$ such that
\begin{equation} \label{Eq 3.4.1}
\sum_{k=0}^{N-1} \tau\left\Vert \frac{S_w^{k+1}-S_w^k}{{\tau}}\right\Vert_{H^{-1}(\Omega)}^2 \leq C, \quad
\sum_{k=0}^{N-1} \tau\left\Vert \frac{\mu_w^{k+1}-\mu_w^k}{{\tau}}\right\Vert_{H^{-1}(\Omega)}^2 \leq C,
\end{equation}
and
\begin{equation} \label{Eqnew 3.4.1}
\sum_{k=0}^{N-1} \tau\left\Vert \mu_w^{k+1}-\mu_w^k\right\Vert_{L^2(\Omega)}^2 \leq C\tau.
\end{equation}
\end{lemma}
\begin{proof}
First, from the definition of $F$, we have
$$
F_{\min}\leq F(S_w)\leq F_{\max},\quad \forall S_w\in [S_{\epsilon},1-S_{\epsilon}],
$$
where $F_{\min}$ and $F_{\max}$ are two constants.
By \eqref{Eq 3.3.1.1} and the fact $a^2+b^2\leq 3a^2+3(a+b)^2$, we can obtain
\begin{align*}
 &\ \frac{\phi_m c_{\min}}{2}\sum_{k=0}^{N-1} \left\Vert S_w^{k+1}-S_w^k\right\Vert_{L^2(\Omega)}^2
 +\frac{\lambda_{\min}K_{\min}}{3}\sum_{k=0}^{N-1} {\tau}\left(\Vert\nabla \mu_w^{k+1}\Vert_{L^2(\Omega)}^2+\Vert\nabla p^{k+1}\Vert_{L^2(\Omega)}^2\right) \\
\leq &\ (F_{\max}-F_{\min})\Vert \phi\Vert_{L^\infty(\Omega)}
+\sum_{k=0}^{N-1} \left(S_w^{k+1}-S_w^{k},\phi\varphi_1^{k+1}\right)
+\varepsilon\sum_{k=0}^{N-1} {\tau}\left(\Vert \mu_w^{k+1}\Vert_{H^1(\Omega)}^2+\Vert p^{k+1}\Vert_{H^1(\Omega)}^2\right)\\
&+C\sum_{k=0}^{N-1}\tau\left(\sum_{\alpha=1,3}\Vert\varphi_\alpha^{k+1}\Vert_{H^1(\Omega)}^2
+\sum_{\alpha=w,n}\Vert q_\alpha^{k+1}\Vert_{L^2(\Omega)}^2
+\sum_{\alpha=2,4} \Vert\varphi_\alpha^{k+1}\Vert^2_{L^\infty(\Gamma_2)}\right),
\end{align*}
where $\varepsilon$ is a positive constant that can be arbitrarily small.
Then following the facts
\begin{align*}
\Vert p^{k+1}\Vert_{H^1(\Omega)}^2\leq C\Vert\nabla p^{k+1}\Vert_{L^2(\Omega)}^2 \quad {\rm and} \quad  \Vert \mu_w^{k+1}\Vert_{L^2(\Omega)}\leq C\Vert \mu_w^{k+1}\Vert_{L^{\infty}(\Omega)}\leq C,
\end{align*}
we have
\begin{align*}
 \sum_{k=0}^{N-1} \left\Vert S_w^{k+1}-S_w^k\right\Vert_{L^2(\Omega)}^2+\sum_{k=0}^{N-1} {\tau}\left(\Vert \mu_w^{k+1}\Vert_{H^1(\Omega)}^2+\Vert p^{k+1}\Vert_{H^1(\Omega)}^2\right)
\leq  C\sum_{k=0}^{N-1}  \left(S_w^{k+1}-S_w^{k},\phi\varphi_1^{k+1}\right)+C.
\end{align*}
For the first term on the right-hand side of the above inequality, it holds from Assumption \ref{assum 3} that
\begin{align*}
 \sum_{k=0}^{N-1}  \left(S_w^{k+1}-S_w^{k},\phi\varphi_1^{k+1}\right)
= &\ \left(\phi, S_w^N\varphi_1^N-S_w^0\varphi_1^0\right)-\sum_{k=0}^{N-1}  \left(\phi S_w^{k},\varphi_1^{k+1}-\varphi_1^{k}\right) \\
\leq &\ C+\sum_{k=0}^{N-1}\int_{t_k}^{t_{k+1}}(\lvert\phi S_w^k\rvert,\lvert\partial_t\varphi_1\rvert) dt \\
\leq &\ C+\Vert \phi\Vert_{L^\infty(\Omega)} \Vert \partial_t \varphi_1 \Vert_{L^1(0,T;L^1(\Omega))} \leq  C,
\end{align*}
which directly yields \eqref{Eq 3.3.1}. Obviously \eqref{Eqnew 3.4.1} can be derived by using \eqref{FderL} and \eqref{Eq 3.3.1}.

For the proof of \eqref{Eq 3.4.1}, since $\varphi_1\in H^1(0,T;H^{-1}(\Omega))$, we have
$$
\sum_{k=0}^{N-1}\tau \left\Vert \frac{\varphi_1^{k+1}-\varphi_1^k}{\tau}\right\Vert_{H^{-1}(\Omega)}\leq
C\Vert\partial_t \phi\Vert_{L^2(0,T;H^{-1}(\Omega))}.
$$
Then by \eqref{F21} it is easy to see that for any test function $\xi\in W$, we have
\begin{align*}
&\ \left(\phi\frac{S_w^{k+1}-S_w^k}{\tau},\xi\right) \\
=&\ -\left(\lambda_w^{k+1} \bm{\mathcal{K}}\nabla\left( \mu_w^{k+1}+p^{k+1}\right), \nabla \xi\right) +\left(q_w^{k+1},\xi\right)+\left\langle\varphi_4^{k+1},\xi\right\rangle_{\Gamma_2}\\
\leq &\ \lambda_{\max}K_{\max}\left\Vert \mu_w^{k+1}+p^{k+1}\right\Vert_{H^1(\Omega)} \Vert \xi\Vert_{H^1(\Omega)}
+C\Vert q_w^{k+1}\Vert_{H^{-1}(\Omega)} \Vert \xi\Vert_{H^1(\Omega)}+\Vert\varphi_4^{k+1}\Vert_{L^\infty(\Gamma_2)} \Vert \xi\Vert_{L^2(\Gamma_2)} \\
\leq &\ C\left(\left\Vert \mu_w^{k+1}+p^{k+1} \right\Vert_{H^1(\Omega)}+\Vert q_w^{k+1}\Vert_{H^{-1}(\Omega)}
+\Vert\varphi_4^{k+1}\Vert_{L^\infty(\Gamma_2)}\right) \Vert \xi\Vert_{H^1(\Omega)}.
\end{align*}
Hence, combining the boundary, we can derive that
\begin{equation} \label{Eq 3.4.1.1.1}
\begin{aligned}
&\ \left\Vert \frac{S_w^{k+1}-S_w^k}{{\tau}}\right\Vert_{H^{-1}(\Omega)} \\
\leq &\ C\left(\left\Vert \mu_w^{k+1}+p^{k+1}\right\Vert_{H^1(\Omega)}
+\Vert q_w^{k+1}\Vert_{H^{-1}(\Omega)}+\Vert\varphi_4^{k+1}\Vert_{L^\infty(\Gamma_2)}+\left\Vert \frac{\varphi_1^{k+1}-\varphi_1^k}{\tau}\right\Vert_{H^{-1}(\Omega)} \right).
\end{aligned}
\end{equation}
Then summing the above inequality over $k$ from $0$ to $N-1$ and combining \eqref{FderL}, we get \eqref{Eq 3.4.1}.
We complete the proof.
\end{proof}

On each subinterval $I_k$, we define the piecewise constant functions
 $$
 S_{w,N}\Big|_{I_k}=S_w^{k+1}, \quad \mu_{w,N}\Big|_{I_k}=\mu_w^{k+1}, \quad p_{N}\Big|_{I_k}=p^{k+1}, \quad DS_{w,N}\Big|_{I_k}=\frac{S_w^{k+1}-S_w^{k}}{\tau}, \quad \varphi_{\alpha,N}\Big|_{I_k}=\varphi_\alpha^{k+1}
 $$
for $\alpha=1,3$, and the piecewise linear function
$$\mu_{w,N,L}\Big|_{I_k}=\frac{t_{k+1}-t}{\tau}\mu_w^{k}+\frac{t-t_k}{\tau}\mu_w^{k+1}.$$
By Lemma \ref{lemma 3.4}, we have
\begin{align*}
\sum_{k=0}^{N-1}\int_{I_k} \left\Vert \mu_{w,N,L}\right\Vert_{H^1(\Omega)}^2 \mathrm{d}t= &\ \sum_{k=0}^{N-1}\int_{I_k} \left\Vert \frac{t_{k+1}-t}{\tau}\mu_w^{k}
+\frac{t-t_k}{\tau}\mu_w^{k+1}\right\Vert_{H^1(\Omega)}^2 \mathrm{d}t \\
\leq &\ \sum_{k=0}^{N-1}\tau \left(\left\Vert \mu_w^{k}\right\Vert_{H^1(\Omega)}^2+\left\Vert \mu_w^{k+1}\right\Vert_{H^1(\Omega)}^2\right)\leq C,
\end{align*}
and $\varphi_{\alpha,N}\to \varphi_{\alpha}$ strongly in $L^2(0,T;H^1(\Omega))$ for $\alpha=1,3$ since $\varphi_{\alpha,N}$ is the piecewise constant interpolant of $\varphi_{\alpha}$ with respect to time variable.

From Lemma \ref{lemma 3.4} and the analysis above, we know that there exists a subsequence (still labeled by $N$)
$$\left\{\left(\mu_{w,N,L},\ \mu_{w,N}-\varphi_{1,N},\ \partial_t \mu_{w,N,L},\ S_{w,N},\ p_{N}-\varphi_{3,N}, \ DS_{w,N}\right)\right\}$$
such that there exists
$\left(\hat{\mu}_w,\ \mu_{w,*},\ \hat{\mu}_w',\ S_{w},\ p_*,\ S_{w}'\right)\in L^2(0,T;H^1(\Omega))\times L^2(0,T;W)\times L^2(0,T;L^2(\Omega))\times L^2(0,T;V)\times L^2(0,T;H^{-1}(\Omega))$
satisfying
\begin{equation}\label{FSweakconvergence}
    \begin{aligned}
        &\ \mu_{w,N,L}\rightharpoonup \hat{\mu}_w,\quad \mu_{w,N}\rightharpoonup \mu_{w,*}+\varphi_1= \mu_w\ \textup{weakly in}\ L^2(0,T;H^{1}(\Omega))\\
        &\ \partial_t\mu_{w,N,L}\rightharpoonup \hat{\mu}_w' \ \textup{weakly in}\ L^2(0,T;H^{-1}(\Omega)),  \\
        &\ S_{w,N}\rightharpoonup S_w\ \textup{weakly in}\ L^2(0,T;L^{2}(\Omega)),\\
        &\ p_{N}\rightharpoonup p_{*}+\varphi_3=p\ \textup{weakly in}\ L^2(0,T;H^{1}(\Omega)), \\
        &\ DS_{w,N}\rightharpoonup S_w'\ \textup{weakly in}\ L^2(0,T;H^{-1}(\Omega)),
    \end{aligned}
\end{equation}
as $N\rightarrow \infty$.
Here, we have the fact (see Appendix C) that
\begin{equation}\label{appendixC}
\hat{\mu}_w'=\partial_t \hat{\mu}_w\quad \text{and}\quad S_w'=\partial_t S_w.
\end{equation}

Since $H^1(\Omega)$ is compactly embedded in $L^2(\Omega)$ and $L^2(\Omega)$ is continuously embedded in $H^{-1}(\Omega)$, then combining Aubin-Lions-Simon lemma (see Corollary 4 in \textup{\cite{Simon1986}}), we can conclude $\mu_{w,N,L}\rightarrow \hat{\mu}_w$ strongly in $L^2(0,T;L^2(\Omega))$ and $\hat{\mu}_w\in C(0,T;L^2(\Omega))$.
\begin{lemma}\label{lemma 3.2.1}
Under assumptions \ref{assum 1}-\ref{assum 4}, we have the subsequences $\mu_{w,N}\rightarrow \mu_w$ and $S_{w,N}\rightarrow S_w$ strongly in $L^2(0,T;L^2(\Omega))$ as $N\rightarrow \infty$ with $S_w=\calS(\mu_w)$.
\end{lemma}
\begin{proof}
From Lemma \ref{lemma 3.4}, by the definition of $\mu_{w,N,L}$ and $\mu_{w,N}$, we obtain
\begin{align*}
    \int_0^T \left\Vert \mu_{w,N,L}-\mu_{w,N}\right\Vert_{L^2(\Omega)}^2 \mathrm{d}t
     =  \sum_{k=0}^{N-1}\int_{I_k} \left\Vert \frac{t_{k+1}-t}{\tau}\left(\mu_{w}^{k+1}-\mu_{w}^{k}\right)\right\Vert_{L^2(\Omega)}^2 \mathrm{d}t
     \leq C\tau,
\end{align*}
which yields $\mu_{w,N}\rightarrow \hat{\mu}_w$ strongly in $L^2(0,T;L^2(\Omega))$ by using $\mu_{w,N,L}\rightarrow \hat{\mu}_w$ strongly in $L^2(0,T;L^2(\Omega))$.
Then combining $\mu_{w,N}\rightharpoonup \mu_w$ weakly in $L^2(0,T;H^{1}(\Omega))$, we obtain $\hat{\mu}_w=\mu_w\in C(0,T;L^2(\Omega))$ and $\mu_{w,N}\rightarrow {\mu}_w$ strongly in $L^2(0,T;L^2(\Omega))$.
On the other hand, we can obtain from the inequality \eqref{FderL} that $S_{w,N}=\calS(\mu_{w,N})\rightarrow \calS(\mu_w)$ strongly in $L^2(0,T;L^{2}(\Omega))$. Thus according to $S_{w,N}\rightharpoonup S_w$ weakly in $L^2(0,T;L^{2}(\Omega))$, we have $S_w=\calS(\mu_w)$.
\end{proof}

Now, based on the above convergence result, we prove the existence of weak solutions in Theorem \ref{thm 2.1}.

\textit{Proof of the existence of the weak solution in Theorem \ref{thm 2.1}.}  Here, we will show that the limit $\left(S_w, \mu_w, p\right)$ of the subsequence $\left(S_{w,N}, \mu_{w,N}, p_N\right)$ is a weak solution of the weak formulation defined in Definition 1.
Selecting $\xi, v\in C_0^\infty(0,T,H^1(\Omega))$ with $\xi^k=\frac{1}{\tau}\int_{I_k}\xi(t)\mathrm{d}t$, $v^k=\frac{1}{\tau}\int_{I_k}v(t)\mathrm{d}t$ in \eqref{equF20}, and letting $N\rightarrow \infty$, we can find from $DS_{w,N}\rightharpoonup \partial_t S_w$
weakly in $L^2(0,T;H^{-1}(\Omega))$ that
\begin{align*}
 \sum_{k=0}^{N-1} {\tau} \left(\phi \frac{{\calS}(\mu_{w}^{k+1})-S_{w}^{k}}{\tau}, \xi^k\right)
\rightarrow  \int_{0}^T\left(\phi \partial_t S_w, \xi\right) \mathrm{d} t.
\end{align*}
Then by $S_{w,N}\rightarrow S_w$ strongly in $L^2(0,T;L^2(\Omega))$, we can obtain from the continuity of $q_\alpha (\alpha=w,n)$ and $\varphi_\alpha (\alpha=2,4)$ that
\begin{align*}
 \sum_{k=0}^{N-1} {\tau} (q_w^{k+1},\xi^k)  \rightarrow \int_0^T (q_w,\xi) \mathrm{d} t, \quad
 \sum_{k=0}^{N-1} {\tau}\left<\varphi_4^{k+1},\xi^k\right>_{\Gamma_2}
\rightarrow &\  \int_0^T \left<\varphi_4,\xi\right>_{\Gamma_2}\mathrm{d} t, \\
 \sum_{k=0}^{N-1} {\tau}\left(q_t^{k+1},v^k\right)  \rightarrow \int_0^T (q_t,v)\mathrm{d} t, \quad
 \sum_{k=0}^{N-1} {\tau} \left<\varphi_{2,4}^{k+1},v^k\right>_{\Gamma_2}
\rightarrow &\  \int_0^T \left<\varphi_{2,4},v\right>_{\Gamma_2}\mathrm{d} t.
\end{align*}
On the other hand, by a simple calculation, we have
\begin{align*}
    &\ \sum_{k=0}^{N-1} {\tau} \left(\lambda_w^{k+1}\bm{\mathcal{K}}\nabla{\mu_{w}^{k+1}},\nabla \xi^k\right)-\int_0^T \left(\lambda_w\bm{\mathcal{K}}\nabla\mu_{w},\nabla \xi\right)\mathrm{d}t  \\
    =&\ \sum_{k=0}^{N-1} \int_{I_k} \left(\lambda_w^{k+1}\bm{\mathcal{K}}\nabla{\mu_{w}^{k+1}},\nabla (\xi^k-\xi)\right) \mathrm{d}t
    +\sum_{k=0}^{N-1} \int_{I_k} \left(\left(\lambda_w^{k+1}-\lambda_w\right)\bm{\mathcal{K}}\nabla{\mu_{w}^{k+1}},\nabla \xi\right) \mathrm{d}t\\
    &\ +\sum_{k=0}^{N-1} \int_{I_k}\left(\lambda_w\bm{\mathcal{K}}\nabla\left(\mu_{w}^{k+1}-\mu_w\right),\nabla \xi\right) \mathrm{d}t  = R_1+R_2+R_3.
\end{align*}
Using the continuity of $\lambda_w$, $S_{w,N}\rightarrow S_w$ strongly in $L^2(0,T;L^2(\Omega))$, and the Lebesgue dominated convergence theorem (see DCT in \cite{Van1968}), we get
$$R_2\rightarrow 0,\quad \textup{as}\ N\rightarrow \infty.$$
As for $R_1$ and $R_3$, according to $\mu_{w,N}\rightharpoonup \mu_w$ weakly in $L^2(0,T;H^{1}(\Omega))$ and $\xi^k=\frac{1}{\tau}\int_{I_k}\xi(t)\mathrm{d}t$, we obtain
$$R_1\rightarrow 0,\quad R_3\rightarrow 0,\quad \textup{as}\ N\rightarrow \infty.$$
Thus,
\begin{align*}
\sum_{k=0}^{N-1} {\tau}\left(\lambda_w^{k+1}\bm{\mathcal{K}}\nabla{\mu_{w}^{k+1}},\nabla \xi^k\right)
 \rightarrow  \int_0^T \left(\lambda_w\bm{\mathcal{K}}\nabla{\mu_w},\nabla \xi\right) \mathrm{d} t.
 \end{align*}
 Similarly, we can deduce that
 \begin{align*}
 \sum_{k=0}^{N-1} {\tau}\left(\lambda_w^{k+1}\bm{\mathcal{K}}\nabla p^{k+1}, \nabla \xi^k \right)
\rightarrow &\ \int_0^T \left(\lambda_w\bm{\mathcal{K}}\nabla p, \nabla \xi \right) \mathrm{d} t, \\
\sum_{k=0}^{N-1} {\tau}\left(\lambda_t^{k+1}\bm{\mathcal{K}}\nabla p^{k+1}, \nabla v^k \right)
\rightarrow &\ \int_0^T \left(\lambda_t\bm{\mathcal{K}}\nabla p, \nabla v \right)\mathrm{d} t.
\end{align*}
 With all ingredients obtained in the above together, it is easy to verify that $(\mu_w, p)$ is a weak solution of the equation \eqref{equF20}, which completes the proof due to the fact $S_w=\calS(\mu_w)$ and Theorem \ref{mainresult}. $\hfill \Box$

\section{Uniqueness of the weak solution}\label{section5}

In order to prove that the weak solution of the weak formulation defined in Definition 1 is unique, we further introduce the following assumption and lemma used in \cite{Evans2022}.
\begin{assum}\label{assum 5}
 Assume that $\lambda_w$, $\lambda_n$, $q_w$, $q_n$, $\varphi_2$ and $\varphi_4$ are Lipschitz continuous with respect to $S_w$. Moreover, we assume the boundary $\varphi_1\in H^1(0,T;W^{-1,\infty}(\Omega))$ and the chemical potential $\mu_w$ is Lipschitz continuous in space and the associated Lipschitz constant belongs to $L^2$ in time, i.e. there exists a constant $C=C(t)$ such that
 \begin{equation}\label{eqassum 5.1}
\left|\mu_w(\bm{x}_1,t)-\mu_w(\bm{x}_2,t)\right| \leq C(t)|\bm{x}_1-\bm{x}_2|, \ \forall \bm{x}_1, \bm{x}_2\in \Omega,
\end{equation}
and $C(t)\in L^2([0,T])$.
\end{assum}

\begin{lemma}\label{lemma 5.2}
\textup{(Characterization of $W^{1,\infty}$, Theorem 4 in Section 8 of Chapter 5 in \cite{Evans2022})} Let $u:\Omega\to \mathbb{R}$ is Lipschitz continuous, then $u\in W^{1,\infty}(\Omega)$.
\end{lemma}

Then we have the following result for the uniqueness of the weak solution. 

\begin{theorem}\label{thm 2.1.1}
Under Assumptions \ref{assum 1}-\ref{assum 5}, the weak solution of the weak formulation defined in Definition \ref{define1} is unique.
\end{theorem}
\begin{proof}
Let $(S_{1},\mu_{1},p_{1})$ and $(S_{2},\mu_{2},p_{2})$ be two of weak solutions for the weak formulation defined in Definition \ref{define1}. By denoting $e_{S}=S_2-S_1$, $e_{\mu}=\mu_2-\mu_1$ and $e_p=p_2-p_1$, and setting $\xi=e_{\mu}$ and $v=e_{p}$, we can directly deduce from \eqref{equF20} and the Cauchy-Schwarz inequality that
\begin{align*}
 &\int_0^T \left(\phi \partial_t e_S, e_{\mu}\right) \mathrm{d} t
 +\int_0^T \left(\lambda_w(\mu_2)\bm{\mathcal{K}}\nabla \left(e_{\mu}+e_p\right),\nabla \left(e_{\mu}+e_p\right) \right) \mathrm{d} t
 +\int_0^T \left(\lambda_n(\mu_2)\bm{\mathcal{K}}\nabla e_p, \nabla e_{p} \right)\mathrm{d} t \\
 =&\ -\int_0^T \left((\lambda_n(\mu_2)-\lambda_n(\mu_1))\bm{\mathcal{K}}\nabla {p_1},\nabla e_{p}\right)\mathrm{d} t
  -\int_0^T \left((\lambda_w(\mu_2)-\lambda_w(\mu_1))\bm{\mathcal{K}}\nabla\left(\mu_1+p_1\right),
   \nabla \left(e_{\mu}+e_p\right) \right)\mathrm{d} t\\
  &+\int_0^T (q_w(\mu_2)-q_w(\mu_1),e_{\mu})\mathrm{d} t
 +\int_0^T \left\langle\varphi_4(\mu_2)-\varphi_4(\mu_1),e_{\mu}\right\rangle_{\Gamma_2} \mathrm{d} t  \\
  &+\int_0^T (q_t(\mu_2)-q_t(\mu_1),e_{p})\mathrm{d} t
 +\int_0^T \left\langle\varphi_{2,4}(\mu_2)-\varphi_{2,4}(\mu_1),e_{p}\right\rangle_{\Gamma_2}\mathrm{d}t \\
 \leq &\ -\int_0^T \left((\lambda_n(\mu_2)-\lambda_n(\mu_1))\bm{\mathcal{K}}\nabla {p_1},\nabla e_{p}\right)\mathrm{d} t
  -\int_0^T \left((\lambda_w(\mu_2)-\lambda_w(\mu_1))\bm{\mathcal{K}}\nabla\left(\mu_1+p_1\right) ,
   \nabla \left(e_{\mu}+e_p\right) \right)\mathrm{d} t\\
 &+C\int_0^T \Vert e_\mu\Vert_{L^2(\Omega)}^2 \mathrm{d} t
 +\varepsilon\left(\Vert e_\mu\Vert_{L^2(0,T;H^1(\Omega))}^2+\Vert e_p\Vert_{L^2(0,T;H^1(\Omega))}^2\right),
 \end{align*}
where $\varepsilon$ is a positive constant that can be arbitrarily small.
Here, the Lipschitz continuity of $q_\alpha \ (\alpha=w,n), \ \varphi_\alpha \ (\alpha=2,4)$ and the trace theorem (see Lemma 16.1 in \cite{Tartar2007}) are utilized in the last inequality above.
Next, using the fact $a^2+b^2\leq 3a^2+3(a+b)^2$ and the estimate (see Appendix D)
\begin{equation}\label{appendixD}
\left\{\begin{aligned}
&  \Vert \mu_1\Vert_{L^2(0,T;W^{1,\infty}(\Omega))}+\Vert \partial_t {\calS}' (\eta)\Vert_{L^2(0,T;W^{-1,\infty}(\Omega))}\leq C,\\
& \int_0^T \left(\nabla p_1,f\nabla z\right)\mathrm{d}t
\leq C\int_0^T \left(\Vert \mu_1\Vert_{W^{1,\infty}(\Omega)}
+\Vert q_t\Vert_{L^\infty(\Omega)}+\Vert \varphi_{2,4} \Vert_{L^\infty(\Gamma_2)}\right)
\Vert f\Vert_{L^2(\Omega)}\Vert\nabla z\Vert_{L^2(\Omega)}\mathrm{d}t,
\end{aligned}
\right.
\end{equation}
and taking $f=e_{\mu}$ and $z=e_p, e_{\mu}$ respectively, we can obtain from the Cauchy-Schwarz inequality and the Lipschitz continuity of $\lambda_\alpha \ (\alpha=w,n)$ that
 \begin{align*}
 &\ \int_0^T \left(\phi \partial_t e_S, e_{\mu}\right) \mathrm{d} t
 +\frac{\lambda_{\min}K_{\min}}{3}\left(\Vert e_\mu\Vert_{L^2(0,T;H^1(\Omega))}^2+\Vert e_p\Vert_{L^2(0,T;H^1(\Omega))}^2\right) \\
  \leq &\ C\int_0^T
  \left(\Vert \mu_1\Vert_{W^{1,\infty}(\Omega)}^2
+\Vert q_t\Vert_{L^\infty(\Omega)}^2+\Vert \varphi_{2,4} \Vert_{L^\infty(\Gamma_2)}^2\right)
 \Vert e_\mu\Vert_{L^2(\Omega)}^2 \mathrm{d} t \\
  &\ +\varepsilon\left(\Vert e_\mu\Vert_{L^2(0,T;H^1(\Omega))}^2+\Vert e_p\Vert_{L^2(0,T;H^1(\Omega))}^2\right),
 \end{align*}
and for the first term on the left-hand side of the inequality above, we have
\begin{align*}
 \int_0^T \left(\phi \partial_t e_S, e_{\mu}\right) \mathrm{d} t
= &\ \int_0^T \left(\phi \partial_t ({\calS}' (\eta) e_{\mu}), e_{\mu}\right) \mathrm{d} t \\
= &\ \frac{1}{2} \int_0^T \partial_t \left(\phi e_{\mu}, {\calS}' (\eta)e_{\mu} \right) \mathrm{d} t
+\frac{1}{2}\int_0^T \left(\phi \partial_t {\calS}' (\eta), e_{\mu}^2\right) \mathrm{d} t \\
\geq &\ \frac{\phi_m S_\epsilon}{2} \Vert e_{\mu}(\bm{x},T)\Vert_{L^2(\Omega)}^2
-\frac{\Vert \phi\Vert_{L^\infty(\Omega)}}{2}\int_0^T \Vert \partial_t {\calS}' (\eta)\Vert_{W^{-1,\infty}(\Omega)} \Vert e_{\mu}^2\Vert_{W^{1,1}(\Omega)} \mathrm{d} t \\
\geq &\ c\Vert e_{\mu}(\bm{x},T)\Vert_{L^2(\Omega)}^2
-C\int_0^T \Vert \partial_t {\calS}' (\eta)\Vert_{W^{-1,\infty}(\Omega)}\left( \Vert \nabla e_{\mu}^2\Vert_{L^1(\Omega)}+\Vert e_{\mu}^2\Vert_{L^1(\Omega)}\right) \mathrm{d} t \\
\geq &\ c\Vert e_{\mu}(\bm{x},T)\Vert_{L^2(\Omega)}^2
-C\int_0^T \Vert \partial_t {\calS}' (\eta)\Vert^2_{W^{-1,\infty}(\Omega)}\Vert e_{\mu}\Vert_{L^2(\Omega)}^2 \mathrm{d} t-\varepsilon\int_0^T \Vert \nabla e_{\mu}\Vert_{L^2(\Omega)}^2 \mathrm{d} t,
\end{align*}
where $\eta$ is between $\mu_1$ and $\mu_2$.
Based on the analysis above, we can get
\begin{equation}\label{Femu}
\begin{aligned}
&\ \Vert e_{\mu}(\bm{x},T)\Vert_{L^2(\Omega)}^2
 +\Vert e_\mu\Vert_{L^2(0,T;H^1(\Omega))}^2+\Vert e_p\Vert_{L^2(0,T;H^1(\Omega))}^2 \\
  \leq &\ C\int_0^T \left(
  \Vert \partial_t {\calS}' (\eta)\Vert_{W^{-1,\infty}(\Omega)}^2+\Vert \mu_1\Vert_{W^{1,\infty}(\Omega)}^2
+\Vert q_t\Vert_{L^\infty(\Omega)}^2+\Vert \varphi_{2,4} \Vert_{L^\infty(\Gamma_2)}^2
  \right)
   \Vert e_\mu\Vert_{L^2(\Omega)}^2 \mathrm{d} t,
\end{aligned}
\end{equation}
which implies $ \Vert e_{\mu}(\bm{x},T)\Vert_{L^2(\Omega)}=0$ by the Gr\"onwall inequality in \cite{Gronwall1919}.  Hence, we obtain $ \Vert e_{\mu}\Vert_{L^\infty(0,T;L^2(\Omega))}=0$ since $T$ can be replaced by any time $t$ in the inequality \eqref{Femu}.
Then we can conclude $\Vert e_\mu\Vert_{L^2(0,T;H^1(\Omega))}^2+\Vert e_p\Vert_{L^2(0,T;H^1(\Omega))}^2=0$ since the inequality \eqref{Femu}, which completes the proof. 
\end{proof}

\section{Regularity of the weak solution}\label{section5}

In this section, since $\lambda_w, \lambda_n, p_c=-\mu_w$ depend on $S_w$, we first introduce the artificial pressure (see \cite{Arbogast1992,Chen2002})
\begin{equation}\label{reglobal}
\psi=p+\int_{S_\epsilon}^{S_w} \left(\frac{\lambda_w}{\lambda_t}\frac{\partial \mu_w}{\partial S_w}\right)(\xi)\mathrm{d}\xi,
\end{equation}
and the complementary pressure (see \cite{Arbogast1992,Chen2002})
\begin{equation}\label{rekirch}
\theta=\int_{S_\epsilon}^{S_w} \left(\frac{\lambda_w\lambda_n}{\lambda_t}\frac{\partial \mu_w}{\partial S_w}\right)(\xi)\mathrm{d}\xi.
\end{equation}
Then the equivalent formulation \eqref{F13} of the model can be deduced to
\begin{equation}\label{Fref13}
\begin{cases}
\phi\partial_t S_w-\nabla\cdot \left(\bm{\mathcal{K}}\nabla \theta+\lambda_w \bm{\mathcal{K}}\nabla \psi\right)=q_w, \\
-\nabla\cdot \left(\lambda_t \bm{\mathcal{K}}\nabla \psi\right)=q_t.
 \end{cases}
\end{equation}

For convenience, throughout this section, we consider the 
complete Neumann boundary conditions
$$
 \begin{cases}
 \lambda_n\bm{\mathcal{K}}\nabla p_n\cdot \mathbf{n} = \varphi_2, &\text{on}~ \partial\Omega \times (0,T], \\
 \lambda_w\bm{\mathcal{K}}\nabla p_w\cdot \mathbf{n} =\varphi_4, &\text{on}~\partial\Omega \times (0,T].
 \end{cases}		
$$
On the other hand, we can follow from \eqref{reglobal} and \eqref{rekirch} that
\begin{equation}\label{Feq5511}
\nabla \psi=\nabla p+\frac{\lambda_w}{\lambda_t}\nabla \mu_w,\quad \nabla \theta=\frac{\lambda_w\lambda_n}{\lambda_t}\nabla \mu_w,
\end{equation}
which yields the facts that
\begin{align*}
\bm{\mathcal{K}}\nabla\theta\cdot \mathbf{n} +\lambda_w\bm{\mathcal{K}}\nabla\psi\cdot \mathbf{n}=\lambda_w\bm{\mathcal{K}}\nabla p_w\cdot \mathbf{n}=\varphi_4, \\
\lambda_t\bm{\mathcal{K}}\nabla\psi\cdot \mathbf{n}=\lambda_n\bm{\mathcal{K}}\nabla p_n\cdot \mathbf{n}+\lambda_w\bm{\mathcal{K}}\nabla p_w\cdot \mathbf{n}=\varphi_{2}+\varphi_{4}=\varphi_{2,4}.
\end{align*}
Therefore, we close the equation \eqref{Fref13} by providing 
the complete Neumann boundary conditions
\begin{equation} \label{Fref13.1N}
 \begin{cases}
 \bm{\mathcal{K}}\nabla\theta\cdot \mathbf{n} +\lambda_w\bm{\mathcal{K}}\nabla\psi\cdot \mathbf{n}= \varphi_4, &\text{on}~ \partial\Omega \times (0,T],\\
 \lambda_t\bm{\mathcal{K}}\nabla\psi\cdot \mathbf{n}=\varphi_{2,4}, &\text{on}~ \partial\Omega \times (0,T].
 \end{cases}		
\end{equation}

In order to analyze the regularity of the weak solution for the new model, we present the following assumptions.
\begin{assum} \label{assum 6}
We assume that $\varphi_2,\ \varphi_4$ and $\lambda_w,\ \lambda_n$ are Lipschitz continuous with respect to $S_w$.
\end{assum}

\begin{assum} \label{assum 7}
We let the domain $\Omega$ be convex with $C^{1,\beta_1}$ boundary for some $\beta_1\in (0,1)$ and assume the initial condition $S_w^0\in H^1(\Omega)\cap C^{0, \beta_2}(\overline\Omega)$ for some $\beta_2\in (0,1)$. Furthermore, we assume
\begin{align*}
\gamma_w, \gamma_n,\gamma_{wn}\in H^1(\Omega),\quad
\bm{\mathcal{K}}\in \bm{\mathcal{W}}^{1,\infty}(\Omega) = [W^{1,\infty}(\Omega)]^{d\times d},
\end{align*}
 where $\bm{\mathcal{K}}\in \bm{\mathcal{W}}^{1,\infty}(\Omega)$ represents that each component of the tensor $\bm{\mathcal{K}}$ belongs to $W^{1,\infty}(\Omega)$.
\end{assum}

Then based on two intermediate variables $\theta$ and $\psi$, we have the following regularity results.
\begin{theorem}\label{thm 5.1}
Let the triple $(S_w,\mu_w,p)$ be the weak solution of the continuous model defined in Definition \ref{define1}. Then under Assumptions \ref{assum 1}-\ref{assum 4}, \ref{assum 6} and \ref{assum 7}, we have
\begin{equation} \label{regular 1}
\begin{cases}
 S_w\in L^\infty(0,T;H^1(\Omega))\cap H^1(0,T;L^2(\Omega)), \\
\mu_w,\ p \in L^2(0,T;H^{2}(\Omega))\cap L^\infty(0,T;H^1(\Omega)).
\end{cases}
\end{equation}
\end{theorem}

\begin{proof}
We will first prove $\psi\in L^\infty(0,T;W^{1,\infty}(\Omega))$. According to Theorem 2.5 in \cite{Chen2002}, we can similarly follow from
$$0 \leq \frac{\lambda_{\min}^2 L_{\min}}{2\lambda_{\max}}\leq \frac{\lambda_w\lambda_n}{\lambda_t}\frac{\partial \mu_w}{\partial S_w}\leq \frac{\lambda_{\max}^2 L_{\max}}{2\lambda_{\min}},\quad S_w^0\in C^{0, \beta_2}(\overline\Omega), \quad S_\epsilon\leq S_w\leq 1-S_\epsilon,$$
the Lipschitz continuous of $\varphi_\alpha, \ \alpha=2,4$ with respect to $S_w$, and $q_w\in L^\infty(0,T;L^\infty(\Omega))$ that
$$S_w\in L^\infty(0,T;C^{0, \beta_0}(\overline\Omega))$$
for some $\beta_0\in (0,\beta_1]\cap (0,\beta_2]$. Then due to the Lipschitz continuity of $\lambda_t=\lambda_w+\lambda_n$ with respect to $S_w$, we obtain $\lambda_t\in L^\infty(0,T;C^{0, \beta_0}(\overline\Omega))$.
By using the regularity theory for elliptic systems such as Theorem 5.21 in \cite{Giaquinta2013} and the facts
$$\lambda_t\in L^\infty(0,T;C^{0, \beta_0}(\overline\Omega)),\quad q_t\in L^\infty(0,T;L^\infty(\Omega)),\quad \varphi_{2,4}\in L^\infty(0,T;L^\infty(\Omega)), $$
we can derive from the second equation of \eqref{Fref13} and that
\begin{equation}\label{Feq51}
\psi\in L^\infty(0,T;W^{1,\infty}(\Omega)).
\end{equation}

Next, we will show
$$\partial_t S_w\in L^2(0,T;L^2(\Omega)),\quad S_w,\ \mu_w,\ p \in L^\infty(0,T;H^1(\Omega)).$$
 By the facts $\gamma_w, \gamma_n,\gamma_{wn}\in H^1(\Omega)$ and $\mu_w\in L^2(0,T;H^1(\Omega))$, we obtain from $S_w=\calS(\mu_w)$ that
\begin{equation}\label{Feq52}
S_w\in L^2(0,T;H^1(\Omega)).
\end{equation}
Then by the Lipschitz continuity of $\lambda_w, \lambda_t$ with respect to $S_w$ and the Rademacher theorem such as Theorem {R} in \cite{Nekvinda1988}, we have the boundedness of
$$\frac{\partial \lambda_w}{\partial S_w}, \ \frac{\partial \lambda_n}{\partial S_w}.$$
Hence, the Lipschitz continuous of $\varphi_\alpha, \ \alpha=2,4$ with respect to $S_w$ and \eqref{Feq51}-\eqref{Feq52} imply
\begin{equation}\label{Feq53}
\varphi_{2},\ \varphi_{4}\in L^2(0,T;H^1(\Omega)),\quad \nabla \lambda_w\cdot \nabla \psi,\ \nabla \lambda_t\cdot \nabla \psi\in L^2(0,T;L^2(\Omega)).
\end{equation}
Then using \eqref{Feq53}, $\lambda_t\geq 2\lambda_{\min}>0$, and $q_t\in L^\infty(0,T;L^\infty(\Omega))$, and combining the regularity theory for elliptic systems (see Theorem 9.15 in \cite{Gilbarg1977}), we can see from the second equation of \eqref{Fref13} that
\begin{equation}\label{Feq54}
\psi \in L^2(0,T;H^2(\Omega)).
\end{equation}
Using \eqref{Feq51}-\eqref{Feq54}, $q_w\in L^\infty(0,T;L^\infty(\Omega))$ and $\bm{\mathcal{K}}\in \bm{\mathcal{W}}^{1,\infty}(\Omega)$, we obtain
\begin{equation}\label{Feq541}
\nabla\cdot \left(\lambda_w \bm{\mathcal{K}}\nabla \psi\right)+q_w\in L^2(0,T;L^2(\Omega)).
\end{equation}
Hence, by $\varphi_4\in L^2(0,T;H^1(\Omega))$ and the Galerkin approximation method (see Chapter 7 in \cite{Evans2022}), the first equation of \eqref{Fref13} yields
\begin{equation}\label{Feq55}
\partial_t S_w\in L^2(0,T;L^2(\Omega)),\quad \theta \in L^\infty(0,T;H^1(\Omega)).
\end{equation}
Then due to the facts that $\gamma_w, \gamma_n,\gamma_{wn}\in H^1(\Omega)$, $\psi\in L^\infty(0,T;W^{1,\infty}(\Omega))$ and $\theta \in L^\infty(0,T;H^1(\Omega))$, \eqref{Feq55} and \eqref{Feq5511} imply
\begin{equation}\label{Feq551}
S_w, \mu_w,p \in L^\infty(0,T;H^1(\Omega)).
\end{equation}


Finally, we demonstrate $\mu_w,\ p\in L^2(0,T;H^2(\Omega))$. Using \eqref{Feq541}-\eqref{Feq55} and $\varphi_2,\ \varphi_4\in L^2(0,T;H^1(\Omega))$, and combining the regularity theory for elliptic systems (see Theorem 9.15 in \cite{Gilbarg1977}), we deduce from the first equation of \eqref{Fref13} that
\begin{equation}\label{Feq56}
\theta \in L^2(0,T;H^2(\Omega)).
\end{equation}
By a simple calculation, we can find that
\begin{align*}
\Delta \theta = &\ \frac{\lambda_w\lambda_n}{\lambda_t}\Delta \mu_w+\nabla\left(\frac{\lambda_w\lambda_n}{\lambda_t}\right)\cdot\nabla \mu_w, \\
= &\ \frac{\lambda_w\lambda_n}{\lambda_t}\Delta \mu_w+\frac{\partial \lambda_w\lambda_n/\lambda_t}{\partial S_w}\frac{\partial \calS}{\partial \mu_w} \nabla \mu_w\cdot\nabla \mu_w.
\end{align*}
Moreover, using the fact (see Problem 10 in Chapter 5 of \cite{Evans2022})
$$
\Vert \nabla \theta\Vert_{L^4(\Omega)}^2\leq \Vert \theta\Vert_{L^\infty(\Omega)}\Vert \Delta \theta\Vert_{L^2(\Omega)},
$$
we have $\nabla \theta\in L^4(0,T;L^{4}(\Omega))$ due to
$$\Vert \theta\Vert_{L^\infty(\Omega)}\leq \frac{\lambda_{\max}^2 L_{\max}}{2\lambda_{\min}},$$
and $\theta \in L^2(0,T;H^2(\Omega))$. Then by \eqref{Feq5511}, we can see that $\nabla \mu_w\in L^4(0,T;L^{4}(\Omega))$, which means
$$\nabla \mu_w\cdot \nabla \mu_w \in  L^2(0,T;L^{2}(\Omega)).$$
Hence, now we can obtain from the regularity theory for elliptic systems (see Theorem 9.15 in \cite{Gilbarg1977}) that $\mu_w\in L^2(0,T;H^2(\Omega))$.
Similarly, we have $p \in L^2(0,T;H^2(\Omega))$,
which concludes the proof.
\end{proof}

\section{Conclusion}\label{section6}

In this paper, using the energy stability estimate and the zeros of a vector field theorem, we prove the existence of a weak solution of the thermodynamically consistent model for incompressible and immiscible two-phase flow in porous media. The uniqueness of the weak solution is obtained by assuming the Lipschitz continuity of the chemical potential rather than the artificial pressure defined in \cite{Chen2001}.
Furthermore, based on Theorem 2.5 in \cite{Chen2002}, we obtain the regularity of the weak solution for the model with complete Neumann boundary conditions by the regularity theory of elliptic PDEs. 
This paper is the first work on the well-posedness and regularity for the thermodynamically consistent model of two-phase flow in porous media \cite{Gao2020,Kou2022}.

\section*{Appendix A.\ \emph{Proof of the inequality \eqref{appendixA}}}

 To show the inequality \eqref{appendixA}, we first introduce a truncation function $H$ satisfying $H(x)=0$ if $x>0$ and $H(x)=1$ if $x\leq 0$, which implies
$$(\mu_w-M_{\min})^-=(\mu_w-M_{\min})H(\mu_w-M_{\min}),\quad \frac{\mathrm{d} H(x)}{\mathrm{d} x}=-\delta(x),$$
where $\delta(x)$ is Dirac delta function (for further details about this function, refer to \cite{Hassani2009}).
Then from the chain rule of differentiation, we can obtain the equality
\begin{align*}
&\ \int_\Omega\partial_t\Big(\big(F(S_w)-F(S_\epsilon)-M_{\min}(S_w-S_\epsilon)\big)H(\mu_w-M_{\min})\Big)\mathrm{d}\bm{x} \\
= &\ \int_\Omega\partial_t \calS(\mu_w) (\mu_w-M_{\min})H(\mu_w-M_{\min})\mathrm{d}\bm{x} \\
&\ -\int_\Omega\big(F(S_w)-F(S_\epsilon)-M_{\min}(S_w-S_\epsilon)\big)\delta(S_w-S_\epsilon)\partial_t \calS(\mu_w)\mathrm{d}\bm{x} \\
= &\ \int_\Omega\partial_t \calS(\mu_w) (\mu_w-M_{\min})H(\mu_w-M_{\min})\mathrm{d}\bm{x}.
\end{align*}
Hence, we can see that
\begin{align*}
&\ \int_0^T \left(\phi \partial_t \calS(\mu_w), (\mu_w-M_{\min})^-\right) \mathrm{d} t \\
= &\ \int_0^T \left(\phi \partial_t S_w, (\mu_w-M_{\min})^-\right) \mathrm{d} t \\
= &\ \int_\Omega \phi \Big(\big(F(S_w(\bm{x},T))-F(S_\epsilon)-M_{\min}(S_w(\bm{x},T)-S_\epsilon)\big)H(\mu_w(\bm{x},T)-M_{\min})\Big) \mathrm{d} \bm{x} , 
\end{align*}
which yields the inequality \eqref{appendixA} due to the fact $\big(F(S_w)-F(S_\epsilon)-M_{\min}(S_w-S_\epsilon)\big)\geq 0$ when $\mu_w\leq M_{\min}$.

\section*{Appendix B.\ \emph{Proof of the estimate \eqref{Eq 3.4.1.2}}}

By a simple calculation, we can see that
\begin{align}
&\ F(S_w^{k+1})-F(S_w^k)+\frac{c_{\min}}{2}\left(S_w^{k+1}-S_w^k\right)^2
-\left(S_w^{k+1}-S_w^k\right)\mu_w^{k+1}    \label{appdFmu} \\
=&\ \sum_{\alpha=w,n} \gamma_\alpha S_\alpha^{k} \left(\ln \left(\frac{S_\alpha^{k+1}}{S_\alpha^k}\right)-
\frac{S_\alpha^{k+1}}{S_\alpha^k}+1\right)
+\left(\gamma_{wn}+\frac{c_{\min}}{2}\right)\left(S_w^{k+1}-S_w^{k}\right)^2 \label{appdFmu2}.
\end{align}
Then the derivative for the function \eqref{appdFmu2} with respect to $S_w^{k+1}$ can be calculated and given as
\begin{align*}
\left(\frac{\gamma_w}{S_w^{k+1}}+\frac{\gamma_n}{1-S_w^{k+1}}-2\gamma_{wn}-c_{\min}\right) \left(S_w^{k}-S_w^{k+1}\right),
\end{align*}
which yields that the function \eqref{appdFmu} reaches its maximum value of $0$ at $S_w^{k+1}=S_w^k$ due to Assumption \ref{assum 2}.
Hence, we obtain the estimate \eqref{Eq 3.4.1.2}.

\section*{Appendix C.\ \emph{Proof of the fact \eqref{appendixC}}}

In the appendix, we first introduce the following result that for $v\in C_0^\infty(0,T;H^1(\Omega))$, we have
\begin{equation}\label{apppendix}
\sum_{k=0}^{N-2} \int_{I_k} \left\Vert \partial_t v-\frac{v^{k+1}-v^k}{\tau}\right\Vert_{H^1(\Omega)}^2\mathrm{d}t\rightarrow 0, \quad \sum_{k=0}^{N-1} \int_{I_k} \Vert v-v^k \Vert_{H^1(\Omega)}^2\mathrm{d}t\rightarrow 0,
\end{equation}
as $N\rightarrow \infty$, where $v^k=\frac{1}{\tau}\int_{I_k}v(t)\mathrm{d}t$. The proof of \eqref{apppendix} is given as follows.
Similar to that $\varphi_{\alpha,N}$ is the piecewise constant interpolant of $\varphi_{\alpha}$ in the $L^2$ space with respect to time, we also have
\begin{align*}
 \sum_{k=0}^{N-1} \int_{I_k} \left\Vert v-v^k\right\Vert_{H^1(\Omega)}^2\mathrm{d}t \rightarrow 0,
\end{align*}
as $N\rightarrow \infty$. Similarly, by using the fact $v^k=\frac{1}{\tau}\int_{I_k}v(t)\mathrm{d}t$ and a simple calculation, we have
\begin{align*}
 \sum_{k=0}^{N-2} \int_{I_k} \left\Vert \partial_t v-\frac{v^{k+1}-v^k}{\tau}\right\Vert_{H^1(\Omega)}^2\mathrm{d}t
= &\ \sum_{k=0}^{N-2} \int_{I_k} \left\Vert \partial_t v-\frac{\int_{I_k}\partial_t v(t+\nu\tau)\mathrm{d}t}{\tau}\right\Vert_{H^1(\Omega)}^2\mathrm{d}t \rightarrow 0,
\end{align*}
as $N\rightarrow \infty$, where $\nu\in [0,1]$. On the other hand, we will also use the fact
\begin{align*}
\left|(S_{w}^0,v^{0})\right|\leq \left\Vert v^{0}\right\Vert_{L^2(\Omega)}= \left\Vert \int_{I_0}\frac{t\partial_t v}{\tau}\mathrm{d}t\right\Vert_{L^2(\Omega)}\leq \tau\left\Vert \partial_{t} v\right\Vert_{L^\infty(0,T;L^2(\Omega))}^2 \rightarrow 0,
\end{align*}
and similarly
\begin{align*}
\left|(S_{w}^N,v^{N-1})\right| \rightarrow 0,
\end{align*}
as $N\rightarrow \infty$, according to $v\in C_0^\infty(0,T;H^1(\Omega))$.

 We first show that $S_w'=\partial_t S_w$. Let $v\in C_0^\infty(0,T;H^1(\Omega))$, then we obtain
\begin{align*}
    \sum_{k=0}^{N-1} \tau \left(\frac{S_{w}^{k+1}-S_{w}^k}{\tau},v^{k}\right) =  -\sum_{k=0}^{N-2} \tau \left(\frac{v^{k+1}-v^k}{\tau},S_{w}^{k+1}\right)+\left(S_{w}^N,v^{N-1}\right)-\left(S_{w}^0,v^{0}\right),
\end{align*}
where $v^k=\frac{1}{\tau}\int_{I_k}v(t)\mathrm{d}t$. Letting $N\rightarrow \infty$ and using the analysis above, we have
\begin{align*}
    \lim\limits_{N\to +\infty}\sum_{k=1}^{N-1} \tau\left(\frac{S_{w}^{k+1}-S_{w}^k}{\tau},v^{k}\right) =  -\int_0^T \left(\partial_t v,S_w\right) \mathrm{d} t,
\end{align*}
since $S_{w,N}\rightarrow S_w$ strongly in $L^2(0,T;L^2(\Omega))$. Then due to
$
DS_{w,N_j}\rightharpoonup S_w',
$
weakly in $L^2(0,T;H^{-1}(\Omega))$, we can derive
\begin{align*}
    \int_0^T \left(S_w',v\right) \mathrm{d} t =  -\int_0^T \left(\partial_t v,S_w\right) \mathrm{d} t,
\end{align*}
which yields from the definition of weak derivative that $S_w'=\partial_t S_w$. Similarly, for $\hat{\mu}_w'$, we also have $\hat{\mu}_w'=\partial_t \hat{\mu}_w$, which concludes the fact \eqref{appendixC}.

\section*{Appendix D.\ \emph{Proof of the estimate \eqref{appendixD}}}
First, we can deduce from Assumption \ref{assum 5} and Lemma \ref{lemma 5.2} that $\mu_1\in L^2(0,T;W^{1,\infty}(\Omega))$.
On the other hand, by the trace theorem (see Lemma 16.1 in \cite{Tartar2007}), \eqref{equF20} implies that
 \begin{align*}
  &\ \int_0^T \left(\lambda_t\bm{\mathcal{K}}\nabla p, \nabla v \right)\mathrm{d} t \\
 =&\ -\int_0^T \left(\lambda_w\bm{\mathcal{K}}\nabla \mu_w, \nabla v \right)\mathrm{d} t
 +\int_0^T (q_t,v)\mathrm{d}t
 +\int_0^T \left<\varphi_{2,4},v\right>_{\Gamma_2} \mathrm{d}t \\
 \leq &\ \lambda_{\max} K_{\max} \int_0^T \Vert \mu_w\Vert_{W^{1,\infty}(\Omega)}\Vert v\Vert_{W^{1,1}(\Omega)}\mathrm{d} t+C\int_0^T \left( \Vert q_t\Vert_{W^{-1,\infty}(\Omega)}+\Vert \varphi_{2,4} \Vert_{L^\infty(\Gamma_2)}\right)\Vert v\Vert_{W^{1,1}(\Omega)}\mathrm{d} t, \\
 \leq &\ C\int_0^T \left(\Vert \mu_w\Vert_{W^{1,\infty}(\Omega)}+\Vert q_t\Vert_{L^\infty(\Omega)}+\Vert \varphi_{2,4} \Vert_{L^\infty(\Gamma_2)}\right)
\Vert\nabla v\Vert_{L^{1}(\Omega)}\mathrm{d} t,
 \end{align*}
 for each $v\in L^2(0,T;V)$. Using the fact $0<\lambda_{\min}\leq \lambda_w$, we set $\nabla v=f\nabla z$ where $f\in L^2(0,T;L^2(\Omega))$ and $z\in L^2(0,T;H^1(\Omega))$ to obatin
$$
\int_0^T \left(\nabla p_1,f\nabla z\right)\mathrm{d}t
\leq C\int_0^T \left(\Vert \mu_1\Vert_{W^{1,\infty}(\Omega)}
+\Vert q_t\Vert_{L^\infty(\Omega)}+\Vert \varphi_{2,4} \Vert_{L^\infty(\Gamma_2)}\right)
\Vert f\Vert_{L^2(\Omega)}\Vert\nabla z\Vert_{L^2(\Omega)}\mathrm{d}t.
$$

Similarly, according to the boundary $\varphi_1\in H^1(0,T;W^{-1,\infty}(\Omega))$, we can obtain from \eqref{equF20} that $\partial_t S_w \in L^2(0,T;W^{-1,\infty}(\Omega))$. Then by the inequality \eqref{FderL}, we can get $ \partial_t \mu_1,  \partial_t \mu_2\in L^2(0,T;W^{-1,\infty}(\Omega))$.
Next, by applying the chain rule of differentiation, we deduce
$$
\partial_t {\calS}' (\eta)={\calS}'' (\eta)\partial_t \eta.
$$
Since $\eta$ is between $\mu_1$ and $\mu_2$, combining the facts $M_{\min}\leq \mu_w\leq M_{\max}$ and $ \partial_t \mu_1,  \partial_t \mu_2\in L^2(0,T;W^{-1,\infty}(\Omega))$, we can see ${\calS}'' (\eta)\partial_t \eta\in L^2(0,T;W^{-1,\infty}(\Omega))$, which implies the estimate \eqref{appendixD}.

\providecommand{\href}[2]{#2}


\begin{thebibliography}{10}

\bibitem{Abels2012}
H. Abels, H. Garcke and G. Grun, Thermodynamically consistent, frame indifferent diffuse interface models for incompressible two-phase flows with different densities, Math. Mod. Meth. Appl. Sci., 22 (2012), 1150013.

\bibitem{Adams2003}
R. A. Adams, and J. J. Fournier, Sobolev spaces, Elsevier, 140, 2003.

\bibitem{Alt1985}
H.W. Alt and E. DiBenedetto, Nonsteady flow of water and oil through inhomogeneous porous media,
Ann. Scuola Norm. Sup. Pisa Cl. Sci., 12 (1985), pp. 335--392.

\bibitem{Amadori2015}
D. Amadori, P. Baiti, A. Corli and E. DalSanto, Global weak solutions for a model of two-phase flow with a single interface, J. Evol. Equations, 15 (2015), pp. 699--726.

\bibitem{Amaziane1996}
B. Amaziane, A. Bourgeat and H. Elamri, Existence of solutions to various rock types model of two-phase flow in porous media, Applicable Analysis, 60 (1996), pp. 121--132.

\bibitem{Arbogast1992}
T. Arbogast, The existence of weak solutions to single porosity and simple dual-porosity models of two-phase incompressible flow, Nonlinear Anal., 19 (1992), pp. 1009--1031.

\bibitem{Arbogast1996}
T. Arbogast and M.F. Wheeler, A nonlinear mixed finite element method for a degenerate parabolic equation arising in flow in porous media, SIAM J. Numer. Anal., 33 (1996), pp. 1669--1687.

\bibitem{Bertsch2003}
M. Bertsch, R.D. Passo and C.J. Van Duijn, Analysis of oil trapping in porous media flow, SIAM J. Math. Anal., 35 (2003), pp. 245--267.

\bibitem{Bourgeat1995}
A. Bourgeat and A. Hidani, A result of existence for a model of two-phase flow in a porous media made of different rock types, Applicable Analysis, 56 (1995), pp. 381--399.

\bibitem{Buzzi2009}
F. Buzzi, M. Lenzinger and B. Schweizer, Interface conditions for degenerate two-phase flow equations in one space dimension, Analysis (Munich), 29 (2009), pp. 299--316.





\bibitem{Cances2009}
C. Canc\`es, Finite volume scheme for two-phase flow in heterogeneous porous media involving capillary pressure discontinuities, ESAIM: Math. Model. Numer. Anal., 43 (2009), pp. 973--1001.

\bibitem{CancesC2009}
C. Canc\`es, T. Gallou\"{e}t and A. Porretta, Two-phase flows involving capillary barriers in heterogeneous porous media, Interfaces Free Bound., 11 (2009), pp. 239--258.

\bibitem{Cances2012}
C. Canc\`{e}s and M. Pierre, An existence result for multidimensional immiscible two-phase flows with discontinuous capillary pressure field, SIAM J. Math. Anal., 44 (2012), pp. 966--992.

\bibitem{Cances2018}
C. Canc\`es, Energy stable numerical methods for porous media flow type problems, Oil \& Gas Science and Technology - Rev. IFP Energies nouvelles, 73 (2018), 78.

\bibitem{Chavent1986}
G. Chavent and J. Jaffr\'e, Mathematical Models and Finite Elements for Reservoir Simulation: Single Phase, Multiphase and Multicomponent Flows through Porous Media, Elsevier, 1986.

\bibitem{Chen2015}
C.Y. Chen and P.Y. Yan, A diffuse interface approach to injection-driven flow of different miscibility in heterogeneous porous media, Phys. Fluids, 27 (2015), 083101.

\bibitem{Chen2019}
H. Chen, J. Kou, S. Sun and T. Zhang, Fully mass-conservative IMPES schemes for incompressible two-phase flow in porous media, Comput. Methods Appl. Mech. Eng., 350 (2019), pp. 641--663.

\bibitem{Chen2001}
Z. Chen, Degenerate two-phase incompressible flow. I: Existence, uniqueness and regularity of a weak solution, J. Differential Equations, 171 (2001), pp. 203--232.

\bibitem{Chen2002}
Z. Chen, Degenerate two-phase incompressible flow. II: Regularity, stability and stabilization, J. Differential Equations, 186 (2002), pp. 345--376.

\bibitem{Chen2006}
Z. Chen, G. Huan and Y. Ma, Computational Methods for Multiphase Flows in Porous Media, SIAM, Philadelphia, PA, USA, 2006.

\bibitem{Cueto2014}
L. Cueto-Felgueroso and R. Juanes, A phase-field model of two-phase Hele-Shaw flow, J. Fluid Mech., 758 (2014), pp. 522--552.

\bibitem{Daim2007}
F.Z. Daim, R. Eymard and D. Hilhorst, Existence of a solution for two-phase flow in porous media: the case that the porosity depends on the pressure, J. Math. Anal. Appl., 326 (2007), pp. 332--351.

\bibitem{Evans2022}
L.C. Evans, Partial Differential Equations, American Mathematical Soc., 2022.

\bibitem{Feng2012}
X. Feng and S. Wise, Analysis of a Darcy-Cahn-Hilliard diffuse interface model for the Hele-Shaw flow and its fully discrete finite element approximation, SIAM J. Numer. Anal., 50 (2012), pp. 1320--1343.

\bibitem{Gao2020}
H. Gao, J. Kou, S. Sun and X. Wang, Thermodynamically consistent modeling of two-phase incompressible flows in heterogeneous and fractured media, Oil \& Gas Science and Technology--Rev. IFP Energies nouvelles, 75 (2020), 32.

\bibitem{Giaquinta2013}
M. Giaquinta and L. Martinazzi, An introduction to the regularity theory for elliptic systems, harmonic maps and minimal graphs, Springer Science \& Business Media, 2013.

\bibitem{Gilbarg1977}
D. Gilbarg and N.S. Trudinger, Elliptic Partial Differential Equations of Second Order, Berlin: springer, 1977.


\bibitem{Gronwall1919}
T. H. Gronwall, Note on the derivatives with respect to a parameter of the solutions of a system of differential equations, Ann. of Math, 20 (1919), pp. 292--296.

\bibitem{Guo2015}
Z. Guo and P. Lin, A thermodynamically consistent phase-field model for two-phase flows with thermocapillary effects, J. Fluid Mech., 766 (2015), pp. 226--271.

\bibitem{Hassani2009}
S. Hassani, Mathematical Methods: For Students of Physics and Related Fields, 2nd ed., Springer, New York, 2009,  pp. 139--170. 

\bibitem{Hoteit2008}
H. Hoteit and A. Firoozabadi, Numerical modeling of two-phase flow in heterogeneous permeable media with different capillarity pressures, Adv. Water Resour., 31 (2008), pp. 56--73.

\bibitem{HoteitH2008}
H. Hoteit and A. Firoozabadi, An efficient numerical model for incompressible two-phase flow in fractured media, Adv. Water Resour., 31 (2008), pp. 891--905.

\bibitem{Hytonen2016}
T. Hyt\"onen, J. van Neerven, M. Veraar and L. Weis, Analysis in Banach Spaces: Volume I: Martingales and Littlewood-Paley Theory, 2016, pp. 1--66. 

\bibitem{Kou2014}
J. Kou and S. Sun, Upwind discontinuous Galerkin methods with conservation of mass of both phases for incompressible two-phase flow in porous media, Numer. Methods Partial Differ. Equ., 30 (2014), pp. 1674--1699.

\bibitem{Kou2022}
J. Kou, X. Wang, S. Du and S. Sun, An energy stable linear numerical method for thermodynamically consistent modeling of two-phase incompressible flow in porous media, J. Comput. Phys., 451 (2022), pp. 110854.

\bibitem{Kou2023}
J. Kou, H. Chen, S. Du and S. Sun, An efficient and physically consistent numerical method for the Maxwell-Stefan-Darcy model of two-phase flow in porous media, Int. J. Numer. Methods Eng., 124 (2023), pp. 546--569.

\bibitem{KouJ2023}
J. Kou, X. Wang, H. Chen and S. Sun, An energy stable, conservative and bounds-preserving numerical method for thermodynamically consistent modeling of incompressible two-phase flow in porous
media with rock compressibility, Int. J. Numer. Methods Eng., 124 (2023), pp. 2589--2617.


\bibitem{KouJ2024_0}
 {J. Kou, A. Salama, H. Chen and S. Sun, Thermodynamically consistent numerical modeling of immiscible two-phase flow in poro-viscoelastic media, Int. J. Numer. Methods Eng., 2024, e7479.}


\bibitem{KouJ2024}
 {J. Kou and X. Wang, Numerical modeling of unsaturated flow in porous media using a thermodynamical approach, Capillarity, 11 (2024), pp. 63--69.}

\bibitem{Kroener1984}
D. Kroener and S. Luckhaus, Flow of oil and water in a porous medium, J. Differential Equations, 55 (1984), pp. 276--288.


\bibitem{Mikelic2010}
A. Mikeli\'c, A global existence result for the equations describing unsaturated flow in porous media with dynamic capillary pressure, J. Differential Equations, 248 (2010), pp. 1561--1577.

\bibitem{Nekvinda1988}
A. Nekvinda, A. and L. Zaj\'i\v cek, A simple proof of the Rademacher theorem, \v Casopis P\v est. Mat., 113 (1988), pp. 337--341.

\bibitem{Hassanizadeh2011}
 {J. Niessner, S. Berg and S. Majid Hassanizadeh, Comparison of two-phase Darcy's law with a thermodynamically consistent approach, Transp. Porous Med., 88 (2011), pp. 133--148.}


\bibitem{Shen2014}
J. Shen and X. Yang, Decoupled energy stable schemes for phase-field models of two-phase complex fluids, SIAM J. Sci. Comput., 36 (2014), pp. B122--B145.

\bibitem{Simon1986}
J. Simon, Compact sets in the space $L^p(O,T; B)$, Ann. Mat. Pura Appl., 146 (1986), pp. 65--96.

\bibitem{Steinle2022}
R. Steinle, T. Kleiner, P. Kumar and R. Hilfer, Existence and uniqueness of nonmonotone solutions in porous media flow, Axioms, 11 (2022), 327.

\bibitem{Tartar2007}
L. Tartar, An Introduction to Sobolev Spaces and Interpolation Spaces, Springer Science \& Business Media, 2007.

\bibitem{Van1968}
H. R. Van der Vaart and E. H. Yen, E, Weak Sufficient Conditions for Fatou's Lemma and Lebesgue's Dominated Convergence Theorem, Math. Mag., 41 (1968), pp. 109--117.

\bibitem{Zhu2019}
G. Zhu, J. Kou, B. Yao, Y.S. Wu, J. Yao and S. Sun, Thermodynamically consistent modelling of two-phase flows with moving contact line and soluble surfactants, J. Fluid Mech., 879 (2019), pp. 327--359.

\bibitem{Zhu2020}
G. Zhu and A. Li, Interfacial dynamics with soluble surfactants: A phase-field two-phase flow model with variable densities, Adv. Geo--Energy Res., 4 (2020), pp. 86--98.

%
%
%
%
%

\end{thebibliography}
\end{document}